\documentclass[11pt]{article}
\usepackage{amsfonts}
\usepackage{graphics}
\usepackage{indentfirst}
\usepackage{cite}
\usepackage{latexsym}
\usepackage{amsmath}
\usepackage{amssymb}
\usepackage[dvips]{epsfig}
\usepackage{amscd}
\usepackage[marginal]{footmisc}
\usepackage[bookmarks,colorlinks,citecolor=blue,linkcolor=red]{hyperref}
\usepackage{mathtools}
\allowdisplaybreaks
\newtheorem{theorem}{Theorem}[section]
\newtheorem{remark}{Remark}[section]

\newtheorem{lemma}[theorem]{Lemma}

\newtheorem{proposition}[theorem]{Proposition}

\def\on{\bar\rho}
\def\O{{\Omega }}

\def\norm[#1]#2{\|#2\|_{#1}}

\def\la{\label}

\def\na{\nabla}
\def\on{\bar\n}

\def\QEDopen{{\setlength{\fboxsep}{0pt}\setlength{\fboxrule}{0.2pt}\fbox{\rule[0pt]{0pt}{1.3ex}\rule[0pt]{1.3ex}{0pt}}}} 
\def\QED{\QEDopen} 
\def\endproof{\hspace*{\fill}~\QED\par\endtrivlist\unskip}

\newcommand{\n}{\rho}

\newcommand{\lm}{\lambda}
\newcommand{\ltwo}{_{L^2}^2}
\newcommand{\pa}{\partial}
\renewcommand{\r}{\mathbb{R}}

\newcommand{\ia}{\int_0^T}
\newcommand{\ga}{\gamma}
\renewcommand{\div}{{\rm div} }
\newcommand{\curl}{{\rm curl} }

\newcommand{\de}{\delta}
\newcommand{\ve}{\varepsilon}

\newcommand{\si}{\sigma}

\newcommand{\ol}{\overline}
\newcommand{\bt}{\begin{theorem}}
\newcommand{\bl}{\begin{lemma}}
\newcommand{\el}{\end{lemma}}
\newcommand{\et}{\end{theorem}}
\newcommand{\bn}{\begin{eqnarray}}
\newcommand{\en}{\end{eqnarray}}
\newcommand{\bnn}{\begin{eqnarray*}}
\newcommand{\enn}{\end{eqnarray*}}
\newcommand{\bnnn}{\begin{eqnarray*}}
\newcommand{\ennn}{\end{eqnarray*}}
\newcommand{\ba}{\begin{aligned}}
\newcommand{\ea}{\end{aligned}}
\newcommand{\be}{\begin{equation}}
\newcommand{\ee}{\end{equation}}

\renewcommand{\thefootnote}{}
\newcommand\blfootnote[1]{%
  \begingroup
  \renewcommand\thefootnote{}\footnote{#1}%
  \addtocounter{footnote}{-1}%
  \endgroup
}
\makeatletter      
\@addtoreset{equation}{section}
\makeatother       

\title{Global Strong Solutions to the Compressible Magnetohydrodynamic Equations with Slip Boundary Conditions in a 3D Exterior Domain}

\author{Yazhou C{\small HEN}, Bin H{\small UANG}, Xiaoding S{\small HI}   \\[3mm] {\normalsize   College of Mathematics and Physics, }\\ {\normalsize  Beijing University of Chemical Technology, Beijing 100029, P. R. China} }

\date{ }

\begin{document}
\maketitle
\blfootnote{Email: chenyz@mail.buct.edu.cn (Y.Chen), abinhuang@gmail.com (B.Huang), shixd@mail.buct.edu.cn (X.Shi)}

\begin{abstract}
In this paper we study the initial-boundary-value problem for the barotropic compressible magnetohydrodynamic system with slip boundary conditions in three-dimensional exterior domain. We establish the global existence and uniqueness of classical solutions to the exterior domain problem with the regular initial data that are of small energy but possibly large oscillations with constant state as far field which could be either vacuum or nonvacuum. In particular, the initial density of such a classical solution is allowed to have large oscillations contain vacuum states. Moreover, the large-time behavior of the solution is also shown.
\end{abstract}



\section{Introduction}
We consider the viscous barotropic compressible magnetohydrodynamic (MHD) equations for isentropic flows in a domain $\Omega\subset\r^{3}$, which can be written as
\begin{equation}\label{CMHD}
\begin{cases}
\rho_t+ \mathop{\mathrm{div}}\nolimits(\rho u)=0,\\
(\rho u)_t+\mathop{\mathrm{div}}\nolimits(\rho u\otimes u)+\nabla P
=\mu \Delta u+(\mu+\lambda)\nabla \mathop{\mathrm{div}}\nolimits u +(\nabla\times H)\times H,\\
H_t -\nabla \times (u \times H)=-\nu \nabla \times (\nabla \times H),
\\
\mathop{\mathrm{div}}\nolimits H=0,
\end{cases}
\end{equation}
where $(x,t)\in\Omega\times (0,T]$, $t\geq 0$ is time, and $x=(x_1,x_2,x_3)$ is the spatial coordinate. The unknown functions $\rho, u=(u^1,u^2,u^3), P=P(\rho),$ and $H=(H^1,H^2,H^3)$ denote the fluid density, velocity, pressure and magnetic field, respectively. Here we consider the barotropic flows with $\gamma$-law pressure $P(\rho)=a\rho^{\gamma}$ ($a>0$ and $\gamma >1$).
The physical constants $\mu$,  $\lambda$ and $\nu$  are shear viscosity, bulk coefficients and resistivity coefficient respectively satisfying
$\mu>0$, $2\mu +{3}\lambda\geq 0$ and $\nu >0$.

In this paper, we are concerned with the global existence of classical solutions of \eqref{CMHD} in the exterior domain of bounded region with slip boundary condition in $\r^3$, which can be regarded as a continuation of our work in \cite{chs2020}. Throughout this paper, let $\mathbb{D}$ be a simply connected bounded domain in $\r^{3}$ with smooth boundary and contained in the ball $B_R \triangleq \{x\in\r^3 |\,|x|<R \}$ for the fixed $R>0$. 
Let $\Omega\subset\r^{3}$ be the exterior domain to $\mathbb{D}$, i.e., $\Omega=\r^3\setminus \bar{\mathbb{D}}$, that is an unbounded domain with smooth boundary $\partial \Omega$. 
In addition, this paper concerns the problem \eqref{CMHD} with the initial data
\begin{equation}\label{initial}
\displaystyle  (\rho,\rho u, H)\big|_{t=0}=(\rho_0, \rho_0 u_0,H_0),\quad \text{in}\,\,\, \Omega,
\end{equation}
and the far-field behavior
\begin{equation}\label{far-b}
\displaystyle (\rho, u, H)\rightarrow ({\rho}_{\infty}, 0, 0), \quad\quad\text{as}\,\, |x|\rightarrow +\infty,
\end{equation}
where ${\rho}_{\infty}\geq 0$ is a given constant.
The boundary condition is supposed by
\begin{align}
& u\cdot n=0,\,\,\,\curl u\times n=0, &\text{on} \,\,\,\partial\Omega, \label{navier-b}\\
& H \cdot n=0,\,\,\,\curl H\times n=0,  &\text{on} \,\,\,\partial\Omega,\label{boundary}
\end{align}
where $n=(n^1,n^2,n^3)$ is the unit outward normal vector to $\partial \Omega$.
The boundary condition \eqref{navier-b} for the velocity presented in this paper can be regarded as a Navier-type slip boundary condition (see e.g., \cite{cl2019}). For the magnetic field, the boundary condition \eqref{boundary} describes that the boundary $\partial \Omega$ is a perfect conductor (see e.g., \cite{djj2013}).

The compressible MHD system \eqref{CMHD} is known to be one of the mathematical models describing the motion of electrically conducting media (cf. gases) in an electromagnetic field and there have been huge literatures devoted to the analysis of the well-posedness and dynamic behavior to the solutions of the system, see, for example, \cite{cw2002,cw2003,djj2013,df2006,fy2008,fy2009,hhpz2017,hw2008,hw2008-1,hw2010,k1984,sh2012,lxz2013,liu2015,lh2015,lsx2016,tg2016,vh1972,wang2003,xh2017,zjx2009,lz2020-mhd,llz2021} and their references. 
Now, we briefly recall some results concerned with well-posedness of solutions for multi-dimensional compressible MHD equations which are more relatively with our problem.
Vol'pert-Hudjaev \cite{vh1972} and Fan-Yu \cite{fy2009} obtained the local existence of  classical solutions to the 3D compressible MHD equations with the initial density is strictly positive or could contain vacuum, respectively. 
Lv-Huang \cite{lh2015} obtained the local existence of classical solutions in $\r^2$ with vacuum as far field density. Tang-Gao \cite{tg2016} obtained the local strong solutions to the compressible MHD equations in a 3D bounded domain with the Navier-slip condition. 
For global existence, Kawashima \cite{k1984} first established the global smooth solutions to the general electro-magneto-fluid equations in two dimensions with non-vacuum.
Hu-Wang \cite{hw2010} proved the global existence of renormalized solutions for general large initial data, also see \cite{hw2008,fy2008} for the non-isentropic compressible MHD equations. 
Recently, Li et al.\cite{lxz2013} established the global existence and uniqueness of classical solutions with constant state as far field in $\r^3$ with large oscillations and vacuum. Hong et al.\cite{hhpz2017} generalized the result for large initial data when $\gamma-1$ and $\nu^{-1}$ are suitably small. Lv et al.\cite{lsx2016} got the global existence of unqiue classical solutions in 2D case.
Recently, we obtained the global classical solutions with vacuum and small energy but possibly large oscillations in a 3D bounded domain with slip boundary condition in \cite{chs2020}. 
Very recently, Liu et al.\cite{llz2021} established the global existence of smooth solutions and the explicit decay rate near a given constant state for 3-D compressible  full MHD with the boundary conditions of Navier-slip for the velocity filed and perfect conduction for the magnetic field in exterior domains.
However, there are no works about the global existence of the strong (classical) solution to the initial-boundary-value problem \eqref{CMHD}-\eqref{boundary} in a 3D exterior domain with initial density containing vacuum, at least to the best of our knowledge. 

The main purpose of this paper is to establish the global well-posedness of classical solutions of compressible MHD system \eqref{CMHD}-\eqref{boundary} in a 3D exterior domain $\Omega$. Since $\Omega$ is no longer bounded, it is distinguishable from our former work \cite{chs2020}. 
Fortunately, we have the Gagliardo-Nirenberg type inequality in the exterior domain (see Lemma \ref{lem-gn}). Moreover, thanks to \cite{NS2004}, the elliptic regularity for the Neumann problem in exterior domain leads us to derive the estimates for the gradient of the effective viscous flux (see Lemma \ref{lem-f-td}), which plays an important role in our analysis. Besides, we also apply the $L^p$-theory for the div-curl system for exterior domains to control $\nabla u$ by means of $\div u$ and $\curl u$ (Lemma \ref{lem-vn}-\ref{lem-high}). In addition, the difficulties caused by the slip boundary still exist and cannot be deal with by the methods in \cite{chs2020} directly. To deal with this difficulty, we adapt the idea of \cite{cl2021} to obtain the boundary estimates (see Lemma \ref{lem-be}), which are used frequently to control the boundary terms in this paper. Furthermore, in order to estimate the derivatives of the solutions, we recall the similar Beale-Kato-Majda-type inequality in the exterior domain (see Lemma \ref{lem-bkm}) to prove the important estimates on the gradients of the density and velocity.
Based on the above analysis, we will establish the well-posedness of classical solutions to the initial-boundary-value problem of compressible MHD system in a 3D exterior domain with lager oscillations and vacuum.

Before formulating our main result, we first explain the notation and conventions used throughout the paper.
For integer $k\geq 1$ and $1\leq q<+\infty$, We denote the standard Sobolev space by $D^{k,q}(\Omega), W^{k,q}(\Omega)$ and $D^k(\Omega)\triangleq D^{k,2}(\Omega), H^k(\Omega)\triangleq W^{k,2}(\Omega)$.
For simplicity, we denote $L^q(\Omega)$, $W^{k,q}(\Omega)$, $H^k(\Omega)$ and $D^k(\Omega)$ by $L^q$, $W^{k,q}$, $H^k$ and $D^k$ respectively, and set
$$\int fdx \triangleq \int_\Omega fdx,\quad \int_0^T\int fdxdt\triangleq\int_0^T\int_\Omega fdxdt. $$
For two $3\times 3$  matrices $A=\{a_{ij}\},\,\,B=\{b_{ij}\}$, the symbol $A\colon  B$ represents the trace of $AB^*$, where $B^*$ is the transpose of $B$, that is,
$$ A\colon  B\triangleq \text{tr} (AB^*)=\sum\limits_{i,j=1}^{3}a_{ij}b_{ij}.$$
Finally, for $v=(v^1,v^2,v^3)$, we denote $\nabla_iv\triangleq(\partial_iv^1,\partial_iv^2,\partial_iv^3)$ for $i=1,2,3,$ and the
material derivative of $v$ by  $\dot v\triangleq v_t+u\cdot\nabla v$.

The initial total energy of \eqref{CMHD} is defined as
\begin{align}\label{c0}
\displaystyle  C_0 =\int_{\Omega}\left(\frac{1}{2}\rho_0|u_0|^2 + G(\rho_0)+\frac{1}{2}|H_0|^2 \right)dx.
\end{align}
where
\begin{align*}
\displaystyle  G(\rho)\triangleq\rho\int_{{\rho}_{\infty}}^{\rho}\frac{P(s)-{P}_{\infty}}{s^{2}} ds,\quad  {P}_{\infty}\triangleq P({\rho}_{\infty}).
\end{align*}

Now we can state our main result, Theorem \ref{th1}, concerning existence of global classical solutions to the problem  \eqref{CMHD}-\eqref{boundary}.
\begin{theorem}\label{th1}
Let $\Omega$ be the exterior domain of a simply connected bounded region $\mathbb{D}$ in $\r^3$ with smooth boundary $\partial\Omega$. For $q\in (3,6)$ and some given constants $M_1,M_2>0$, and $\bar{\rho}\geq{\rho}_{\infty}+1$ , and the initial data $(\rho_0,u_0,H_0)$ satisfy the boundary conditions \eqref{navier-b}-\eqref{boundary} and
\begin{gather}
0\leq\rho_0\leq\bar{\rho},\quad
(\rho_0-{\rho}_{\infty},P(\rho_0)-{P}_{\infty})\in H^2\cap W^{2,q}, \label{dt1}\\
 (u_0, H_0)\in  D^1 \cap D^2 ,\quad \rho_0|u_0|^2+G(\rho_0)+|H_0|^2\in L^1,\quad \mathop{\mathrm{div}}\nolimits H_0=0,\label{dt2}\\
 \|\nabla u_0\|_{L^2}\leq M_1,\quad \|\nabla H_0\|_{L^2}\leq M_2,\label{dt-s}
\end{gather}
and the compatibility condition
\begin{align}\label{dt3}
\displaystyle  -\mu\triangle u_0-(\mu+\lambda)\nabla \mathop{\mathrm{div}}\nolimits u_0 + \nabla P(\rho_0)- (\nabla \times H_0) \times H_0 = \rho_0^{\frac{1}{2}}g,
\end{align}
for some  $ g\in L^2.$ Moreover, $\rho\in L^{3/2}$ when ${\rho}_{\infty}=0$.
Then there exists a positive constant $\ve$ depending only on  $\mu$, $\lambda$, $\nu$, $\ga$, $a$, ${\rho}_\infty$, $\bar{\rho}$,   $M_1$ and $M_2$  such that for the initial energy $C_0$ as in \eqref{c0} if
\begin{equation*}
\displaystyle C_0\leq \ve,
\end{equation*}
then the system \eqref{CMHD}-\eqref{boundary} has a unique global classical solution $(\rho,u,H)$ in $\Omega\times(0,\infty)$ satisfying
\begin{align}\label{esti-rho}
\displaystyle  0\le \rho(x,t)\le 2\bar{\rho},\quad  (x,t)\in \Omega\times(0,\infty),
\end{align}
\begin{equation}\label{esti-uh}
\begin{cases}
(\rho-{\rho}_{\infty},P-{P}_{\infty})\in C([0,\infty);H^2 \cap W^{2,q} ),\\
\nabla u\in C([0,\infty);H^1 )\cap  L^\infty_{\rm loc}(0,\infty;H^2\cap W^{2,q}),\\
u_t\in L^{\infty}_{\rm loc}(0,\infty; D^1\cap D^2)\cap H^1_{\rm loc}(0,\infty; D^1),\\
H \in C([0,\infty);H^2)\cap  L^\infty_{\rm loc}(0,\infty; H^4),\\
H_t\in C([0,\infty);L^2)\cap H^1_{\rm loc}(0,\infty; H^1)\cap L^\infty_{\rm loc}(0,\infty; H^2).	
\end{cases}
\end{equation}
Furthermore,  for all $r\in (2,\infty)$ if ${\rho}_{\infty}>0$ and $r\in (\gamma,\infty)$ if ${\rho}_{\infty}=0$, we have the following large-time behavior
\begin{align}\label{esti-t}
\displaystyle \lim_{t \rightarrow \infty } \big( \|\rho(\cdot,t)-{\rho}_{\infty}\|_{L^r}+\|(\rho^\frac{1}{8}u)(\cdot,t)\|_{L^4}+\|\nabla u(\cdot,t)\|_{L^2}+\|\nabla H(\cdot,t)\|_{L^2}\big)=0.
\end{align}
\end{theorem}

Then, when ${\rho}_{\infty}>0$ and the initial density contains vacuum state, we can deduce the following large-time behavior of the gradient of the density.
\begin{theorem}\label{th2}
Under the conditions of Theorem \ref{th1}, assume further that ${\rho}_{\infty}>0$ and there exists some point $x_0\in \Omega$ such that $\rho_0(x_0)=0.$  Then the unique global classical solution $(\rho,u,H)$ to the problem \eqref{CMHD}-\eqref{boundary} obtained in
Theorem \ref{th1}  satisfies that for any $r_1>3,$
\begin{align}\label{esti-2}
\displaystyle \lim_{t \rightarrow \infty } \|\nabla\rho (\cdot,t)\|_{L^{r_1}}=\infty.
\end{align}
\end{theorem}

\begin{remark}\label{rem:1} From Sobolev embedding theorem and \eqref{esti-uh}$_1$ with $q>3$, the solution obtained in Theorem \ref{th1} becomes a classical one away from the initial time. As we known, this is the first result concerning the global existence of classical solutions to the compressible MHD system in a 3D exterior domain with large oscillations and vacuum.
\end{remark}

\begin{remark}\label{rem:2}
When we consider the following generalized slip boundary for the velocity field:
\begin{align*}\label{navi1} u\cdot n=0,\,\,\,\curl u\times n=-Au \,\,\,&\text{on} \,\,\,\partial\Omega, \end{align*}
and assume that the $3\times 3$ symmetric matrix $A$ is smooth and positive semi-definite, and even if the restriction on $A$ is relaxed to $A\in W^{2,6}$ and the negative eigenvalues of $A$ (if exist) are small enough, in particular, set $A = B-2D(n)$, where $B\in W^{2,6}$ is a positive semi-definite $3\times 3$ symmetric matrix, Theorem \ref{th1} and \ref{th2} will still hold provided that $2\mu+3\lambda>0$. This can be achieved by a similar way as in \cite{cl2019,chs2020}.
\end{remark}

\begin{remark}\label{rem:3} For the magnetic field, we also can subject to the Dirichlet condition
\begin{equation*}
\displaystyle  H =0,\,\,\,\text{on} \,\,\,\partial\Omega,
\end{equation*}
or the insulating boundary condition (see \cite{gll2006})
\begin{equation*}
\displaystyle  H \times n =0,\,\,\,\text{on} \,\,\,\partial\Omega.
\end{equation*}
After some slight modification of the proof in this paper, Theorem \ref{th1} and \ref{th2} will still hold.
\end{remark}

The rest of the paper is organized as follows.
In Section \ref{se2}, we review some known lemmas and derive the elementary energy estimates and some key a priori estimates that we use intensively in this paper.
Section \ref{se3} is devoted to deriving the necessary time-independent lower-order estimates and time-dependent higher-order estimates, which can guarantee the local classical solution to be a global classical one. Finally, the proof of Theorem \ref{th1}-\ref{th2} will be completed in Section \ref{se5}.

\section{Preliminaries}\label{se2}
In this section, we list some known facts and elementary inequalities that are used extensively in this paper. We also derive the elementary energy estimates for the system \eqref{CMHD}-\eqref{boundary} and some key a priori estimates.

\subsection{Some basic inequalities and lemmas}\label{appendix-a}
We first state the following Gagliardo-Nirenberg type inequality in the exterior domain (see \cite{cm2004}).
\begin{lemma}\label{lem-gn}
Assume that $\Omega$ is an exterior domain of some simply connected domain $\mathbb{D}$ in $\r^3$. For  $p\in [2,6],\,q\in(1,\infty),$ and $ r\in  (3,\infty),$ there exist two generic constants $C>0$ which may depend  on $p$, $q$ and $r$ such that for any  $f\in H^1({\O }) $ and $g\in  L^q(\O )\cap D^{1,r}(\O), $ such that
\be\label{g1}\|f\|_{L^p(\O)}\le C \|f\|_{L^2}^{(6-p)/(2p)}\|\na
f\|_{L^2}^{(3p-6)/(2p)},\ee
\be\label{g2}\|g\|_{C\left(\ol{\O }\right)} \le C
\|g\|_{L^q}^{q(r-3)/(3r+q(r-3))}\|\na g\|_{L^r}^{3r/(3r+q(r-3))}.
\ee
\end{lemma}

Next, we give the following Zlotnik inequality (see \cite{zlo2000}), which will be used to get the uniform (in time) upper bound of the density.
\begin{lemma}\label{lem-z}
Suppose the function $y$ satisfy
\bnn y'(t)= g(y)+b'(t) \mbox{  on  } [0,T] ,\quad y(0)=y^0, \enn
with $ g\in C(R)$ and $y, b\in W^{1,1}(0,T).$ If $g(\infty)=-\infty$
and \be\label{a100} b(t_2) -b(t_1) \le N_0 +N_1(t_2-t_1)\ee for all
$0\le t_1<t_2\le T$
  with some $N_0\ge 0$ and $N_1\ge 0,$ then
\bnn y(t)\le \max\left\{y^0,\overline{\zeta} \right\}+N_0<\infty
\mbox{ on
 } [0,T],
\enn
where $\overline{\zeta} $ is a constant such
that \be\label{a101} g(\zeta)\le -N_1 \quad\mbox{ for }\quad \zeta\ge \overline{\zeta}.\ee
\end{lemma}

The following two lemmas are given in \cite[Theorem 3.2]{Von1992} and \cite[Theorem 5.1]{lhm2016}.
\begin{lemma}   \label{lem-vn}
Assume that $\Omega$ is an exterior domain of some simply connected domain $\mathbb{D}$ in $\r^3$ with $C^{1,1}$ boundary. For $v\in D^{1,q}(\Omega)$ with $v\cdot n=0$ on $\partial\Omega$, it holds that
$$\|\nabla v\|_{L^{q}(\Omega)}\leq C(\|\div v\|_{L^{q}(\Omega)}+\|\curl v\|_{L^{q}(\Omega)})\,\,\,for\,\, any\,\, 1<q<3,$$
and
$$\|\nabla v\|_{L^q(\Omega)}\leq C(\|\div v\|_{L^q(\Omega)}+\|\curl v\|_{L^q(\Omega)}+\|\nabla v\|_{L^2(\Omega)})\,\,\,for\,\, any\,\, 3\leq q<+\infty.$$
\end{lemma}
\begin{lemma}\label{lem-curl}
Assume that $\Omega$ is an exterior domain of some simply connected domain $\mathbb{D}$ in $\r^3$ with $C^{1,1}$ boundary. For any $v\in W^{1,q}(\Omega)\,\,(1<q<+\infty)$ with $v\times n=0$ on $\partial\Omega$, it holds that
$$\|\nabla v\|_{L^q(\Omega)}\leq C(\|v\|_{L^q(\Omega)}+\|\div v\|_{L^q(\Omega)}+\|\curl v\|_{L^q(\Omega)}).$$
\end{lemma}

Moreover, we have the following conclusion (see \cite{cl2021}).
\begin{lemma}\label{lem-high}
Assume that $\Omega$ is an exterior domain of some simply connected domain $\mathbb{D}$ in $\r^3$ with $C^{1,1}$ boundary. For $v\in D^{k+1,p}(\Omega)\cap D^{1,2}(\Omega)$ for some $k\geq 0$, $p\in[2,6]$ with $v\cdot n=0$ or $v\times n=0$ on $\partial\Omega$ and $v(x,t)\rightarrow0$ as $|x|\rightarrow\infty$, there exists a positive constant $C=C(q,k,\mathbb{D})$ such that
$$\|\nabla v\|_{W^{k,p}(\Omega)}\leq C(\|\div v\|_{W^{k,p}(\Omega)}+\|\curl v\|_{W^{k,p}(\Omega)}+\|\nabla v\|_{L^2(\Omega)}).$$
\end{lemma}

Besides, similar to \cite{BKM1984}, we need a Beale-Kato-Majda type inequality with respect to the slip boundary condition \eqref{navier-b} which can be found in \cite{cl2019}.
\begin{lemma}\label{lem-bkm}
For $3<q<\infty$, assume that $u\cdot n=0$ and $\curl u\times n=0$ on $\partial\Omega$, $\nabla u\in W^{1,q}$, then there is a constant  $C=C(q)$ such that  the following estimate holds
\bnn\ba
\|\na u\|_{L^\infty}\le C\left(\|{\rm div}u\|_{L^\infty}+\|\curl u\|_{L^\infty} \right)\ln(e+\|\na^2u\|_{L^q})+C\|\na u\|_{L^2} +C .
\ea\enn
\end{lemma}

Finally, we have the following local existence of classical solution of \eqref{CMHD}-\eqref{boundary}, which can be proven in a similar manner as that in \cite{fy2009,tg2016}.
\begin{lemma}\label{lem-local}
Let $\Omega$ be as in Theorem \ref{th1} and assume that the initial date $(\rho_0,u_0,H_0)$ satisfies the conditions \eqref{dt1}, \eqref{dt2} and \eqref{dt3}. Then there exist a small time $T_0>0$ and a unique classical solution $(\rho,u,H)$ of the system \eqref{CMHD}-\eqref{boundary} in $\r^3 \times (0,T_0]$, satisfying that $\rho \geq 0$, and that for $\tau \in (0,T_0)$,
\begin{equation}\label{esti-uh-local}
\begin{cases}
(\rho-{\rho}_{\infty},P-{P}_{\infty})\in C([0,T_0);H^2 \cap W^{2,q} ),\\
\nabla u\in C([0,T_0);H^1 )\cap  L^\infty(\tau,T_0;H^2\cap W^{2,q}),\\
u_t\in L^\infty(\tau,T_0; D^1 \cap D^2)\cap H^1(\tau,T_0; D^1),\\
H \in C([0,T_0);H^2)\cap  L^\infty(\tau,T_0; H^4),\\
H_t\in C([0,T_0);L^2)\cap H^1(\tau,T_0; H^1)\cap  L^\infty(\tau,T_0; H^2).	
\end{cases}
\end{equation}
\end{lemma}

\subsection{Elementary energy estimates}\label{appendix-b}
In the following, let $T>0$ be a fixed time and $(\rho,u,H)$ be a smooth solution to \eqref{CMHD}-\eqref{boundary} on $\Omega \times (0,T]$. We derive the elementary energy estimates for the system \eqref{CMHD}-\eqref{boundary} and some key a priori estimates which are frequently applied later.
First, we rewrite \eqref{CMHD} in the following form:
\begin{equation}\label{CMHD1}
\begin{cases}
\rho_t+ \mathop{\mathrm{div}}\nolimits(\rho u)=0,\\
\rho u_t+\rho u \cdot \nabla u - (\lambda\! +\! 2\mu)\nabla\div u+\mu\nabla\!\times\!\omega + \nabla(P\!-\!{P}_{\infty})=H \!\cdot\! \nabla H+\nabla \frac{|H|^2}{2},\\
H_t+u \cdot \nabla H- H \cdot \nabla u+ H \div u= -\nu \nabla \times \curl H,
\\
\mathop{\mathrm{div}}\nolimits H=0,
\end{cases}
\end{equation}
where $\omega \triangleq \nabla\times u, \curl H \triangleq \nabla\times H$. In view of \eqref{far-b}, \eqref{navier-b} and \eqref{boundary}, multiplying $\eqref{CMHD1}_1 $ by $G'(\rho)$, $\eqref{CMHD1}_2$ by $u$ and $\eqref{CMHD1}_3$ by $H$ respectively, integrating by parts over $\Omega$, summing them up, we obtain
\begin{align}\label{m2}
\displaystyle  &\left(\int \Big(G(\rho)+\frac{1}{2}\rho |u|^{2}+\frac{1}{2}|H|^{2}\Big)dx\right)_t + (\lambda + 2\mu)\int(\div u)^{2}dx \nonumber \\
 & + \mu\int|\omega|^{2}dx 
 + \nu\int|\curl H|^{2}dx =0,
\end{align}
which, integrated over $(0,T)$, leads to the following elementary energy estimates.
\begin{lemma}\label{lem-basic}
 Let $(\rho,u,H)$ be a smooth solution of \eqref{CMHD}-\eqref{boundary} on $\O \times (0,T]$. Then 
\begin{align}
\displaystyle  &\sup_{0\le t\le T}
\left(\frac{1}{2}\|\sqrt{\rho}u\|_{L^2}^2+\|G(\rho)\|_{L^1}+\frac{1}{2}\|H\|_{L^2}^2\right)
\nonumber \\
 &+ \int_0^{T}(\lambda+2\mu)\|\div u\|_{L^2}^2 +\mu\|\omega\|_{L^2}^2+\nu \|\curl H\|_{L^2}^2)dt \le C_0.\label{basic1}
\end{align}
\end{lemma}

\begin{remark}\label{rem:energy}
According to Lemma \ref{lem-vn}, it follows from \eqref{navier-b} and \eqref{boundary} that
\begin{align}
\displaystyle & \|\nabla u\|_{L^2}\leq C(\|\div u\|_{L^2}+\|\omega\|_{L^2}),\label{tdu1}\\
& \|\nabla H\|_{L^2}\leq C\|\curl H\|_{L^2}.\label{tdh1}
\end{align}
Besides, it is easy to check that
\begin{equation}\label{grho}
\begin{cases}
G(\rho)=\frac{1}{\gamma-1}P(\rho),& \text{if} \quad {\rho}_{\infty}=0,\\
\displaystyle  C^{-1}(\rho-{\rho}_{\infty})^{2}\leq G(\rho)\leq C(\rho-{\rho}_{\infty})^{2},& \text{if}\quad {\rho}_{\infty}>0,\, 0\leq \rho \leq 2 \bar{\rho}.
\end{cases}
\end{equation}
Then \eqref{basic1} together with \eqref{tdu1}-\ref{grho}, for $0\leq \rho \leq 2 \bar{\rho}$, yields
\begin{align}\label{basic2}
\displaystyle \sup_{0\le t\le T}\|\rho-{\rho}_{\infty}\|_{L^r}^r+\int_{0}^{T}(\|\nabla u\|_{L^{2}}^{2}+\|\nabla H\|_{L^{2}}^{2})dt\leq CC_{0},
\end{align}
where $r\in [2, \infty)$ if ${\rho}_{\infty}>0$ and $r\in [\gamma, \infty)$ if ${\rho}_{\infty}=0$.
\end{remark}

Similarly to the compressible Navier-Stokes equations, the effective viscous flux
\begin{align}\label{flux}
F\triangleq(\lambda+2\mu)\text{div}u-(P-{P}_{\infty})-\frac{1}{2}|H|^2,
\end{align}
plays an important role in our following analysis. More precisely, we derive some priori estimates for $F$, $\omega$ and $\nabla u$, which will be frequently applied later.
\begin{lemma}\label{lem-f-td}
Let $(\rho,u,H)$ be a smooth solution of \eqref{CMHD}-\eqref{boundary} on $\O \times (0,T]$. Then for any $p\in[2,6],$ there exists a positive constant $C$ depending only on $p$, $\mu$ and $\lambda$ such that
\begin{align}
& \|\nabla F\|_{L^p}\leq C(\|\rho\dot{u}\|_{L^p}+\|H\! \cdot \nabla H\|_{L^p}),\label{tdf1}\\
& \|\nabla\omega\|_{L^p}\leq  C(\|\rho\dot{u}\|_{L^p}+\|H \cdot \nabla H\|_{L^p}+\|\rho\dot{u}\|_{L^2}+\|H \cdot \nabla H\|_{L^2}+\|\nabla u\|_{L^2}),\label{tdxd-u1}\\
& \|F\|_{L^p}\leq C(\|\rho\dot{u}\|_{L^2}+\|H \!\cdot \!\nabla H\|_{L^2})^{\frac{3p-6}{2p}}(\|\nabla u\|_{L^2}+\|P\!-\!{P}_{\infty}\|_{L^2}+\|H\|^2_{L^4})^{\frac{6-p}{2p}},\label{f-lp}\\
& \|\omega\|_{L^p} \leq C(\|\rho\dot{u}\|_{L^2}+\|H \cdot \nabla H\|_{L^2})^{\frac{3p-6}{2p}}\|\nabla u\|_{L^2}^{\frac{6-p}{2p}}+C\|\nabla u\|_{L^2},\label{xdu1}\\
& \|\nabla u\|_{L^p}\leq C(\|\rho\dot{u}\|_{L^2}+\|H \!\cdot \!\nabla H\|_{L^2}+\|P\!-\!{P}_{\infty}\|_{L^6}+\||H|^2\|_{L^6})^{\frac{3p-6}{2p}}\|\nabla u\|_{L^2}^{\frac{6-p}{2p}}+C \|\nabla u\|_{L^2}.\label{tdu2}
\end{align}
\end{lemma}
\begin{proof}First, \eqref{CMHD1}$_2$ and \eqref{flux} yields that
\begin{equation}\label{2m2}
\displaystyle  \rho\dot{u}=\nabla F-\mu\nabla\times\omega+H \cdot \nabla H.
\end{equation}
By \eqref{navier-b}, one can find that the viscous flux $F$ satisfies
$$\int\nabla F\cdot\nabla\eta dx=\int(\rho\dot{u}-H \cdot \nabla H ) \cdot\nabla\eta dx,\,\,\forall\eta\in C^{\infty}(\r^3),$$
It follows from \cite[Lemma 4.27]{NS2004}, for $q\in (1,\infty),$ that
\begin{align}\label{tdf2}
\displaystyle  \|\nabla F\|_{L^q}\leq C(\|\rho\dot{u}\|_{L^q}
+\|H \cdot \nabla H\|_{L^q}),
\end{align}
which gives  \eqref{tdf1}. Besides, for any integer $k\geq 1$,
\begin{equation}\label{tdfk}
\|\nabla\! F\|_{W^{k,q}}\!\leq\! C(\|\rho\dot{u}\|_{W^{k,q}}+\|H\cdot \nabla H\|_{W^{k,q}}).
\end{equation}
Also, noticing that $\omega\times n=0$ on $\partial\Omega$ and $\mathop{\mathrm{div}}\nolimits \omega=0$, by Lemma \ref{lem-curl} and \eqref{2m2}-\eqref{tdfk}, we get
\begin{align}
\displaystyle  \|\nabla\omega\|_{L^q}\leq C(\|\nabla\times\omega\|_{L^q}+\|\omega\|_{L^q}) \leq C(\|\rho\dot{u}\|_{L^q}+\|H \cdot \nabla H\|_{L^q}+\|\omega\|_{L^q}),\label{tdxd-u2}
\end{align}
and for any integer $k\geq 1$,
\begin{align}\label{tdxdk}
&\quad\|\nabla\omega\|_{W^{k,q}}\leq C(\|\nabla\times\omega\|_{W^{k,q}}+\|\omega\|_{L^2})\nonumber \\
&\leq C(\|\rho\dot{u}\|_{W^{k,q}}+\|H \cdot \nabla H\|_{W^{k,q}}+\|\rho\dot{u}\|_{L^2}+\|H \cdot \nabla H\|_{L^2}+\|\nabla u\|_{L^2}).
\end{align}
By Sobolev's inequality and \eqref{tdxd-u2}, for $p\in[2,6]$, it follows that
\begin{align*}
\|\nabla\omega\|_{L^p}
&\le C(\|\rho\dot{u}\|_{L^p}+\|H \cdot \nabla H\|_{L^p}+\|\rho\dot{u}\|_{L^2}+\|H \cdot \nabla H\|_{L^2}+\|\omega\|_{L^2}),
\end{align*}
which implies \eqref{tdxd-u1}.

Now, in view of the Gagliardo-Nirenberg type inequality \eqref{g1} and \eqref{tdf1}, one can deduce that for $p\in[2,6]$,
\begin{align}\label{f-lp1}
\|F\|_{L^p}\leq& C\|F\|_{L^2}^{\frac{6-p}{2p}}\|\nabla F\|_{L^2}^{\frac{3p-6}{2p}}\nonumber \\
\leq& C(\|\rho\dot{u}\|_{L^2}+\|H\! \cdot\! \nabla H\|_{L^2})^{\frac{3p-6}{2p}}(\|\nabla u\|_{L^2}+\|P\!-\!{P}_{\infty}\|_{L^2}+\|H\|^2_{L^4})^{\frac{6-p}{2p}},
\end{align}
similarly, by \eqref{g1} and \eqref{tdxd-u2},
\begin{align}\label{xdu-lp1}
\|\omega\|_{L^p}\leq& C\|\omega\|_{L^2}^{\frac{6-p}{2p}}\|\nabla \omega\|_{L^2}^{\frac{3p-6}{2p}}\nonumber \\
\leq& C(\|\rho\dot{u}\|_{L^2}+\|H \cdot \nabla H\|_{L^2}+\|\nabla u\|_{L^2})^{\frac{3p-6}{2p}}\|\nabla u\|_{L^2}^{\frac{6-p}{2p}}+C\|\nabla u\|_{L^2},
\end{align}
and we arrive at \eqref{f-lp} and \eqref{xdu1}.
Finally, combining \eqref{g1}, \eqref{tdu1}, \eqref{f-lp} and \eqref{xdu1} yields \eqref{tdu2} holds and  the proof is finished.
\end{proof}

\begin{remark} 
From Lemma \ref{lem-high} and Lemma \ref{lem-f-td}, we can get the higher order estimates of $\nabla u$, which will be devoted to giving higher order estimates in Section \ref{se4}. More precisely, we can get the estimates of $\|\nabla^{2}u\|_{L^p}$ and $\|\nabla^{3}u\|_{L^p}$ for $p\in[2,6]$ by Lemma \ref{lem-high}, \eqref{tdf2} and \eqref{tdfk}, for $p\in[2,6]$,
\begin{align}\label{2tdu}
&\quad\|\nabla^{2}u\|_{L^p}\leq C(\|\div u\|_{W^{1,p}}+\|\omega\|_{W^{1,p}}+ \|\nabla u\|_{L^2})\nonumber \\
&\leq C(\|\rho\dot{u}\|_{L^p}+\|H \cdot \nabla H\|_{L^p}+\|\nabla P\|_{L^p}+\|P-{P}_{\infty}\|_{L^p} +\||H|^2\|_{L^p}+\|\nabla H \cdot H\|_{L^p}\nonumber \\
&\quad +\|\rho\dot{u}\|_{L^2}+\|H \cdot \nabla H\|_{L^2}+\|P\!-\!{P}_{\infty}\|_{L^2}+\|H\|^2_{L^4}+\|\nabla u\|_{L^2}),
\end{align}
and
\begin{align}\label{3tdu}
&\quad\|\nabla^{3}u\|_{L^p}\leq C(\|\div u\|_{W^{2,p}}+\|\omega\|_{W^{2,p}}+ \|\nabla u\|_{L^2})\nonumber \\
&\leq C(\|\rho\dot{u}\|_{W^{1,p}}+\|H \cdot \nabla H\|_{W^{1,p}}+\|P-{P}_{\infty}\|_{W^{2,p}}+\||H|^2\|_{W^{2,p}}\nonumber \\
&\quad +\|\rho\dot{u}\|_{L^2}+\|H \cdot \nabla H\|_{L^2}+\|P\!-\!{P}_{\infty}\|_{L^2}+\|H\|^2_{L^4}+\|\nabla u\|_{L^2}).
\end{align}
Moreover, noticing that $H \cdot n=0, \curl H \times n=0$ on $\partial \Omega$, by Lemma \ref{lem-high}, for any integer $k\geq 1$, $p\in [2,6]$,  we obtain
\begin{align}\label{tdhk}
\|\nabla H\|_{W^{k,p}}&\leq C (\|\curl H\|_{W^{k,p}}+\|\nabla H\|_{L^2}) \nonumber \\
 &\leq C(\|\curl^2 H\|_{W^{k-1,p}}+\|\curl H \|_{L^p}+\|\nabla H\|_{L^2}),
\end{align}
where $\curl^2 H \triangleq \curl\curl  H$ and we have used the fact $\div\curl H=0$.
\end{remark}

To this end, we give the following boundary estimates which will be used frequently later.
\begin{lemma}\label{lem-be}
Assume that $\Omega$ is an exterior domain of the simply connected domain $\mathbb{D}$ in $\r^3$ with smooth boundary $\partial \Omega$ and $\bar{\mathbb{D}}\subset B_R$. Let $u\in D^1$ with $u \cdot n=0$ on $\partial \Omega$. It holds for $f\in D^1$,
\begin{align}
&\int_{\partial \Omega}f u \cdot \nabla u \cdot n dS\leq C \|\nabla f\|_{L^2}\|\nabla u\|_{L^2}^2,\label{b-e1}\\
&\int_{\partial \Omega}u \cdot \nabla f dS\leq C \|\nabla f\|_{L^2}\|\nabla u\|_{L^2}.\label{b-e4}
\end{align}
Moreover, for $\nabla u\in L^2 \cap L^4, f\in D^1 \cap D^2 $,
\begin{align}
&\int_{\partial \Omega}f u \cdot \nabla u \cdot \nabla n \cdot u dS\leq  C \|\nabla f\|_{L^6}\|\nabla u\|_{L^2}^3+C\|\nabla f\|_{L^2}\|\nabla u\|_{L^2}\|\nabla u\|_{L^4}^2.\label{b-e3}
\end{align}
\end{lemma}
\begin{proof}
We adapt the ideas due to Cai-Li \cite{cl2019}. Since $u\cdot n=0$ on $\partial\Omega$, it follows that
\begin{align}\label{bdd2}
\displaystyle u\cdot\nabla u\cdot n=-u\cdot\nabla n\cdot u,\quad \mbox{ on } \partial \Omega.
\end{align}
Furthermore, in order to get the boundary estimates, one can extend the unit normal vector $n$ to $\Omega$ such that $n\in C^3(\Omega)$ and $n \equiv 0$ outside $B_{2R}$. Denote $\tilde{\Omega}\triangleq B_{2R}\cap \Omega $, it follows that
\begin{align*}
&\int_{\partial \Omega}f u \cdot \nabla u \cdot n dS=-\int_{\partial \Omega}f u \cdot \nabla n \cdot u dS \nonumber \\
\leq& C (\|f u \cdot \nabla n \cdot u\|_{L^1}+\|\nabla(f u \cdot \nabla n \cdot u)\|_{L^1}) \nonumber \\
\leq& C (\|f\|_{L^6(\tilde{\Omega})}\|u\|_{L^6(\tilde{\Omega})}^2+\|\nabla f\|_{L^2(\tilde{\Omega})}\|u\|_{L^6(\tilde{\Omega})}^2+\|f\|_{L^6(\tilde{\Omega})}\|\nabla u\|_{L^2(\tilde{\Omega})}\|u\|_{L^6(\tilde{\Omega})}) \nonumber \\
\leq & C \|\nabla f\|_{L^2}\|\nabla u\|_{L^2}^2,
\end{align*}
which yields \eqref{b-e1}. Similarly,
\begin{align}
&\int_{\partial \Omega}f^2 u \cdot \nabla u \cdot n dS=-\int_{\partial \Omega}f u \cdot \nabla n \cdot u dS \nonumber \\
\leq& C (\|f^2 u \cdot \nabla n \cdot u\|_{L^1}+\|\nabla(f^2 u \cdot \nabla n \cdot u)\|_{L^1}) \nonumber \\
\leq& C (\|f\|_{L^6(\tilde{\Omega})}^2\|u\|_{L^6(\tilde{\Omega})}^2+\|f\|_{L^6(\tilde{\Omega})}\|\nabla f\|_{L^2(\tilde{\Omega})}\|u\|_{L^6(\tilde{\Omega})}^2+\|f\|_{L^6(\tilde{\Omega})}^2\|\nabla u\|_{L^2(\tilde{\Omega})}\|u\|_{L^6(\tilde{\Omega})}) \nonumber \\
\leq & C \|\nabla f\|_{L^2}^2\|\nabla u\|_{L^2}^2.\label{b-e2}
\end{align}
Moreover, if $f\in L^\infty$, it is easy to check that
\begin{align}
&\int_{\partial \Omega}f u \cdot \nabla u \cdot n dS\leq C \|f\|_{L^\infty}\|\nabla u\|_{L^2}^2.\label{b-e5}
\end{align}

Next, $u\cdot n=0$ on $\partial\Omega$ implies
\begin{align}\label{bdd3}
\displaystyle u=u^{\perp}\times n,\quad \mbox{ on } \partial \Omega,
\end{align}
where $u^{\perp}\triangleq -u\times n$ on $\partial\Omega$.  Noticing that $\div (\nabla f \times v)=\nabla \times v \cdot \nabla f$, we have
\begin{align*}
&\int_{\partial \Omega}u \cdot \nabla f dS=\int_{\partial \Omega}u^{\perp}\times n \cdot \nabla f dS=\int_{\partial \Omega} \nabla f\times u^{\perp}\cdot n dS \nonumber \\
=&\int \div(\nabla f\times u^{\perp})dx=\int \nabla\times u^{\perp}\cdot \nabla f dx
\leq C \|\nabla f\|_{L^2}\|\nabla u\|_{L^2},
\end{align*}
which yields \eqref{b-e4}. Similarly,
\begin{align*}
&\int_{\partial \Omega}f u \cdot \nabla u \cdot \nabla n \cdot udS=\int_{\partial \Omega}f (u^{\perp}\times n) \cdot \nabla u \cdot \nabla n \cdot u dS \nonumber \\
=&\int_{\partial \Omega} f(u^{\perp}\times n)\cdot \nabla u \cdot \nabla n \cdot u dS
=\int_{\partial \Omega} f (\nabla u \cdot \nabla n \cdot u) \times u^{\perp} \cdot n  dS \nonumber \\
=&\int \div(f (\nabla u \cdot \nabla n \cdot u) \times u^{\perp})dx \nonumber \\
=&\int ((\nabla u \cdot \nabla n \cdot u) \times u^{\perp})\cdot \nabla f dx+\int f (\nabla\times u^{\perp})\cdot \nabla u \cdot \nabla n \cdot u dx \nonumber \\
\leq& C \|\nabla f\|_{L^6}\|\nabla u\|_{L^2}^3+\|\nabla f\|_{L^2}\|\nabla u\|_{L^2}\|\nabla u\|_{L^4}^2,
\end{align*}
which yields \eqref{b-e3} and finishes the proof of Lemma \ref{lem-be}.
\end{proof}

\section{\label{se3} A priori estimates}
In this section, we will establish the necessary time-independent lower-order estimates and time-dependent higher-order estimates, which can guarantee the local classical solution to be a global classical one. Let $T>0$ be a fixed time and $(\rho,u,H)$ be a smooth solution to \eqref{CMHD}-\eqref{boundary} on $\Omega \times (0,T]$  with the initial data $(\rho_0,u_0,H_0)$ satisfying \eqref{dt1}-\eqref{dt-s}.

\subsection{\label{se3-1} Time-independent lower order estimates}
In this subsection, we will derive the time-independent a priori bounds of the solutions of the problem \eqref{CMHD}-\eqref{boundary}.

Set $\si=\si(t)\triangleq\min\{1,t \},$ we define
\begin{align}
&  A_1(T) \triangleq \sup_{\mathclap{0\le t\le T}}
\sigma \left( \|\nabla u\|_{L^2}^2 \!+\! \|\nabla H\|_{L^2}^2 \right) \!+\! \int_0^{T} \!\! \sigma (\|\sqrt{\rho}\dot{u}\|_{L^2}^2 \!+\! \|\curl^2 H\|_{L^2}^2 \!+\! \|H_t\|_{L^2}^2)dt,\label{As1}\\
& A_2(T)  \triangleq \sup_{\mathclap{0\le t\le T}}\sigma^2(\|\sqrt{\rho}\dot{u}\|_{L^2}^2\!+\!\|\curl^2 H\|_{L^2}^2\!+\!\|H_t\|_{L^2}^2)\!+\!\!\int_0^{T}\!\!\sigma^2(\|\nabla\dot{u}\|_{L^2}^2\!+\!\|\nabla H_t\|_{L^2}^2)dt,\label{As2}\\
& A_3(T) \triangleq\sup_{  0\le t\le T   } \|H\|_{L^3}^3,\label{As3}\\
& A_4(T) \triangleq\sup_{\mathclap{0\le t\le T}}\sigma^{\frac{1}{4}}\left(\|\nabla u\|_{L^2}^2\!+\!\|\nabla H\|_{L^2}^2\right)\!+\!\!\int_0^{T}\!\!\!\sigma^{\frac{1}{4}}(\|\sqrt{\rho}\dot{u}\|_{L^2}^2 \!+\! \|\curl^2 H\|_{L^2}^2 \!+\!\|H_t\|_{L^2}^2)dt,\label{As4} \\
& A_5(T) \triangleq\sup_{  0\le t\le T   }\|\rho^{\frac{1}{3}}u\|_{L^3}^3,\label{As5}
\end{align}
where $\dot{v}=v_t+u \cdot \nabla v$ is the material derivative.

Now we will give the following key a priori estimates in this section, which guarantees the existence of a global classical solution of \eqref{CMHD}--\eqref{boundary}.
\begin{proposition}\label{pr1} Assume that initial data $(\rho_0,u_0,H_0)$ satisfies \eqref{dt1}-\eqref{dt-s}. Let $(\rho,u,H)$  is a smooth solution of \eqref{CMHD}-\eqref{boundary}  on $\Omega\times (0,T] $ satisfying
\begin{equation}\label{key1}
\begin{cases}
\sup\limits_{\Omega\times [0,T]}\rho\le 2\bar{\rho},\quad
A_1(T) + A_2(T) \le 2C_0^{\frac{1}{2}},\\
A_3(T)\leq 2C_0^{\frac{1}{9}},\quad A_4(\sigma(T))+A_5(\sigma(T))\leq 2C_0^{\frac{1}{9}},
\end{cases}
\end{equation}
then there exists a  positive constant  $\ve$ depending only on $\mu$, $\lambda$, $\nu$, $a$, $\ga$, ${\rho}_{\infty}$, $\bar{\rho}$,   $M_1$ and $M_2$ such that
 \begin{equation}\label{key2}
\begin{cases}
\sup\limits_{\Omega\times [0,T]}\rho\le {\frac{7\bar{\rho}}{4}},\quad
A_1(T) + A_2(T) \le C_0^{\frac{1}{2}},\\
A_3(T)\leq C_0^{\frac{1}{9}},\quad A_4(\sigma(T))+A_5(\sigma(T))\leq C_0^{\frac{1}{9}},
\end{cases}
\end{equation}
provided $C_0\le \ve.$
\end{proposition}
\begin{proof} Proposition \ref{pr1} is a consequence of  Lemmas \ref{lem-h3}, \ref{lem-a3}-\ref{lem-brho} below.
\end{proof}

In what follows, we denote by $C$ or $C_i\ (i=1,2,\cdots)$ the generic positive constants which may depend on $\mu , \lambda,$  $\nu,  \ga ,  a ,$ ${\rho}_{\infty},  \bar{\rho}$,  $M_1$  and $M_2$ but are independent on $T>0$. We also use $C(\alpha)$ to emphasize that $C$ depends on $\alpha$.

\begin{remark}\label{rem-4} Under the assumption \eqref{key1}, it is easy to show that if $C_0\leq 1$, there is a positive constant $C$ such that
\begin{align}\label{basic3}
\displaystyle \int_0^T(\|\nabla u\|_{L^{2}}^{4}+\|\nabla H\|_{L^{2}}^{4})dt\leq CC_{0}^{\frac{2}{9}},
\end{align}
which will be used frequently later.
\end{remark}

The following Lemmas \ref{lem-h3}-\ref{lem-brho} will be proven under the same assumptions as in Proposition \ref{pr1}. First, we give the estimate of $A_3(T)$.
\begin{lemma}\label{lem-h3}
Under the conditions of Proposition \ref{pr1}. Then there is a positive constant
  $\ve_1 $ depending on $\mu,$ $\lambda,$ $\nu,$ $a$, $\gamma$, ${\rho}_{\infty}$, $\bar{\rho}$ and $M_2$  such that
  \begin{align}\label{h-l3}
  \displaystyle A_3(T)\leq C_0^{\frac{1}{9}},
  \end{align}
provided $C_0\leq \ve_1$.
\end{lemma}
\begin{proof}
Multiplying \eqref{CMHD1}$_3$ by $3|H|H$ and integrating by parts over $\Omega$, we have
\begin{align}\label{h3-1}
\frac{d}{dt}\|H\|^3_{L^3}=&-3\nu \int \curl H \cdot \curl (|H|H) dx
+3 \int |H|H \cdot \nabla u \cdot H dx -2 \int |H|^3 \div u dx \nonumber \\
\leq & C\|H\|_{L^\infty} \|\nabla H\|_{L^2}^2+ C \|\nabla u\|_{L^2}\|H\|_{L^6}^3 \nonumber \\
\leq & C \|\nabla H\|^{5/2}_{L^2}\|\curl^2  H\|_{L^2}^{\frac{1}{2}}+C \|\nabla H\|^{2}_{L^2}+ C \|\nabla H\|^{4}_{L^2}+C \|\nabla u\|_{L^2}^4,
 \end{align}
which together with \eqref{key1} and \eqref{basic3} indicates that
\begin{align}\label{h3-2}
&\displaystyle \sup_{0\le t\le T}\|H\|_{L^{3}}^{3} \nonumber \\
\leq & \|H_0\|_{L^{3}}^{3}+C \int_0^T \|\nabla H\|^{5/2}_{L^2}\|\curl^2  H\|_{L^2}^{\frac{1}{2}}dt +C C_0+ C C_0^{\frac{2}{9}} \nonumber \\
\leq & \|H_0\|_{L^{3}}^{3}+C \int_0^{\sigma(T)}
\left(\sigma^{\frac{1}{4}}\|\nabla H\|_{L^2}^2\right)^{5/4}\left(\sigma^{\frac{1}{4}}\|\curl^2  H\|_{L^2}^2\right)^{1/4}\sigma^{-\frac{3}{8}} dt \nonumber \\
&+C\|\nabla H\|_{L^2} \int_{\sigma(T)}^T \left(\|\nabla H\|_{L^2}^2\right)^{3/4}\left(\sigma\|\curl^2  H\|_{L^2}^2\right)^{1/4}dt +C C_0+ C C_0^{\frac{2}{9}} \nonumber \\
\leq & C_1C_0^{\frac{1}{6}}
\end{align}
where in the last inequality we have used the simple fact
\begin{align}\label{h3-3}
\displaystyle  \|H_0\|_{L^{3}}^{3} \leq C \|H_0\|_{L^{2}}^{\frac{3}{2}}\|\nabla H_0\|_{L^2}^{\frac{3}{2}} \leq C(M_2)C_0^{\frac{3}{4}}.
\end{align}
Thus it follows from \eqref{h3-3} that \eqref{h-l3} holds provided
$C_0 \leq \ve_1 \triangleq \min \{1, C_1^{-18}\}.$
The proof of Lemma \ref{lem-h3} is completed.
\end{proof}

\begin{lemma}\label{lem-tdh}
Under the conditions of Proposition \ref{pr1}. It holds
\begin{align}\label{a01}
 \displaystyle  A_1(T) \le  C C_0 + C\int_0^{T}\sigma \|\nabla u\|_{L^3}^3dt.
\end{align}
\end{lemma}
\begin{proof}
The proof of this lemma proceeds in two steps. First, let us first consider the short-time estimate of $H$. Multiplying \eqref{CMHD1}$_3$ by $H$ and integrating by parts over $\Omega$, by \eqref{boundary}, \eqref{g1}, \eqref{tdh1}, we have
\begin{align*}
\displaystyle \left(\frac{1}{2}\|H\|_{L^2}^2\right)_t+\nu \|\curl H\|_{L^2}^2\leq
\|\nabla u\|_{L^2}\|H\|^2_{L^4}\leq \frac{\nu}{2}\|\curl H\|_{L^2}^2+C\|\nabla u\|^4_{L^2}\|H\|^2_{L^2},
\end{align*}
which together with \eqref{tdh1}, \eqref{basic3} and Gronwall inequality gives
\begin{align}\label{tdh-2}
\displaystyle \sup_{0\le t\le T}\|H\|_{L^2}^2+\int_0^T\|\nabla H\|_{L^2}^2dt \leq
C\|H_0\|^2_{L^2}.
\end{align}
By Lemma \ref{lem-f-td}, one easily deduces from \eqref{CMHD1}$_3$ and \eqref{boundary} that
\begin{align}\label{tdh-3}
\displaystyle & \left(\frac{\nu}{2}\|\curl H\|_{L^2}^2\right)_t+\nu^2 \|\curl^2  H\|_{L^2}^2+\|H_t\|^2_{L^2}
\leq C(\|\nabla u\|^2_{L^2}+\|\nabla u\|^4_{L^2})\|\nabla H\|_{L^2}^2,
\end{align}
using \eqref{tdh1}, \eqref{basic3} and Gronwall inequality, we get
\begin{align}\label{tdh-4}
  \displaystyle \sup_{0\le t\le T  }\|\nabla H\|_{L^2}^2+\int_0^{T}
 \left(\|\curl^2 H\|_{L^2}^2+\|H_t\|_{L^2}^2\right)dt \leq C\|\nabla H_0\|^2_{L^2}.
  \end{align}
Besides, multiplying \eqref{tdh-3} by $\sigma$ and integrating it over $(0,T)$, by \eqref{basic3} and \eqref{tdh-2}, we obtain
 \begin{align}\label{tdh}
  \displaystyle \sup_{0\le t\le T  }\left(\sigma\|\nabla H\|_{L^2}^2\right)+\int_0^{T} \sigma
 \left(\|\curl^2 H\|_{L^2}^2+\|H_t\|_{L^2}^2\right)dt \leq CC_0,
  \end{align}

Next, it remains to show the estimate of $u$ and we follow the same plan as used in \cite{chs2020,lxz2013}. We focus on the boundary terms and give the sketch of the proof.
Let $m\ge 0$ be a real number which will be determined later. Multiplying $\eqref{CMHD1}_2 $ by $\sigma^m \dot{u}$ and then integrating the reslting equality over
$\Omega$ lead to
\begin{align}\label{I0}
\int \sigma^m \rho|\dot{u}|^2dx &
= -\int\sigma^m \dot{u}\cdot\nabla Pdx + (\lambda+2\mu)\int\sigma^m \nabla\div u\cdot\dot{u}dx \nonumber\\
&\quad - \mu\int\sigma^m \nabla\times\omega\cdot\dot{u}dx
+\int\sigma^m (H \cdot \nabla H- \nabla|H|^2/2)\cdot\dot{u}dx \nonumber\\
& \triangleq I_1+I_2+I_3+I_4.
\end{align}
First, by $\eqref{CMHD}_1$ and Lemma \ref{lem-f-td}, a direct calculation gives
\begin{align}\label{I1}
I_1 
= & -\int\sigma^m u_{t}\cdot\nabla(P-{P}_{\infty})dx
- \int\sigma^m u\cdot\nabla u\cdot\nabla Pdx \nonumber\\
= & \left(\int\sigma^m(P-{P}_{\infty})\,\div u\, dx\right)_{t} - m\sigma^{m-1}\sigma'\int(P-{P}_{\infty})\,\div u\,dx \nonumber\\
&+ \int\sigma^{m}P\nabla u:\nabla u dx + (\gamma-1)\int\sigma^{m}P(\div u)^{2}dx - \int_{\partial\Omega}\sigma^{m}Pu\cdot\nabla u\cdot n ds\nonumber\\
\leq &\left(\int\sigma^m(P-{P}_{\infty})\,\div u\, dx\right)_{t} + C \|\nabla u \|_{L^2}^2+ Cm\sigma^{m-1}\sigma'C_0,
\end{align}
where we have used \eqref{b-e5} with $f=P$ to deal with the boundary term in the second equality.
Similarly, by \eqref{bdd2}, it indicates that
\begin{align}\label{I20}
I_2 
& = (\lambda+2\mu)\int_{\partial\Omega}\sigma^m\div u\,(\dot{u}\cdot n)ds - (\lambda+2\mu)\int\sigma^m\div u\,\div \dot{u}dx  \nonumber\\
& = (\lambda+2\mu)\int_{\partial\Omega}\sigma^m\div u\,(u\cdot\nabla u\cdot n)ds - \frac{\lambda+2\mu}{2}\left(\int\sigma^{m}(\div u)^{2}dx\right)_{t} \nonumber\\
&\quad +\frac{\lambda+2\mu}{2}\int\sigma^{m}(\div u)^{3}dx- (\lambda+2\mu)\int\sigma^m\div u\,\nabla u:\nabla udx  \nonumber\\
&\quad + \frac{m(\lambda+2\mu)}{2}\sigma^{m-1}\sigma'\int(\div u)^{2}dx
\end{align}
For the boundary term on the righthand side of \eqref{I20}, applying Lemma \ref{lem-be} and \eqref{tdu1}, we obtain
\begin{align}\label{I21}
& (\lambda+2\mu)\int_{\partial\Omega}\sigma^m\div u\,(u\cdot\nabla u\cdot n)ds \nonumber \\
= & \int_{\partial\Omega}\sigma^m Fu\cdot\nabla u\cdot nds+\int_{\partial\Omega}\sigma^m(P-{P}_{\infty})u\cdot\nabla u\cdot nds+\int_{\partial\Omega}\sigma^m \frac{|H|^2}{2}u\cdot\nabla u\cdot nds \nonumber \\
\leq & C\sigma^m(\|\nabla F\|_{L^{2}}\|\nabla u\|_{L^{2}}^2+\|\nabla u\|_{L^{2}}^{2}+\|\nabla H\|_{L^{2}}^2\|\nabla u\|_{L^{2}}^{2})\nonumber \\
\leq & C\sigma^m(\|\rho\dot{u}\|_{L^2}+\|H \cdot \nabla H\|_{L^2})\|\nabla u\|_{L^{2}}^2+C\sigma^m\|\nabla u\|_{L^{2}}^2+C\sigma^m\|\nabla H\|_{L^{2}}^2\|\nabla u\|_{L^{2}}^2 \nonumber \\
\leq & \frac{1}{2}\sigma^m\|\sqrt{\rho}\dot{u}\|_{L^2}^2+C \sigma^m\|\curl^2  H\|_{L^2}^2+C\sigma^m(\|\nabla u\|_{L^{2}}^2+\|\nabla H\|_{L^{2}}^2)(\|\nabla u\|_{L^{2}}^2+1),
\end{align}
where we have used
\begin{align}\label{h-tdh}
\displaystyle \|H \cdot \nabla H\|_{L^2}\leq C \|H\|_{L^3}\|\nabla H\|_{L^6} \leq C C_0^{\frac{1}{27}}(\|\curl^2  H\|_{L^2}+\|\nabla H\|_{L^2}).
\end{align}
Therefore,
\begin{align}\label{I2}
\displaystyle  I_2 \leq &  - \frac{\lambda+2\mu}{2}\left(\int\sigma^{m}(\div u)^{2}dx\right)_t+C\sigma^{m}\|\nabla u\|_{L^{3}}^{3}+\frac{1}{2}\sigma^m\|\sqrt{\rho}\dot{u}\|_{L^2}^2+C \sigma^m\|\curl^2  H\|_{L^2}^2\nonumber \\
& +C\sigma^m(\|\nabla u\|_{L^{2}}^2+\|\nabla H\|_{L^{2}}^2)\|\nabla u\|_{L^{2}}^2+C(\|\nabla u\|_{L^{2}}^{2}+\|\nabla H\|_{L^{2}}^{2}).
\end{align}
Next, by \eqref{navier-b}, a straightforward computation shows that
\begin{align}\label{I3}
\displaystyle  I_3 
& = -\frac{\mu}{2}\left(\int\sigma^{m}|\omega|^{2}dx\right)_t + \frac{\mu m}{2}\sigma^{m-1}\sigma'\int|\omega|^{2}dx   \nonumber\\
& \quad - \mu\int\sigma^{m}(\nabla u^{i}\times\nabla_i u)\cdot\omega dx + \frac{\mu}{2}\int\sigma^{m}|\omega|^{2}\,\div udx\nonumber\\
& \leq -\frac{\mu}{2}\left(\int\sigma^{m}|\omega|^{2}dx\right)_t + Cm\sigma^{m-1}\sigma'\|\nabla u\|_{L^{2}}^{2} + C\sigma^{m}\|\nabla u\|_{L^{3}}^{3}.
\end{align}
Finally, by \eqref{boundary} and \eqref{tdhk}, we have
\begin{align}\label{I4}
I_4 
 = & \left(\int\sigma^m (H \cdot \nabla H- \nabla|H|^2/2)\cdot u dx\right)_t-m\sigma^{m-1}\sigma'\int (H \cdot \nabla H- \nabla|H|^2/2)\cdot u dx \nonumber \\
      &+\int\sigma^m \big((H \otimes H)_t:\nabla u- (|H|^2/2)_t\div u\big)dx +\int\sigma^m (H \cdot \nabla H- \nabla|H|^2/2)\cdot u \cdot \nabla u dx \nonumber \\
 \leq & \left(\int\sigma^m (H \cdot \nabla H- \nabla|H|^2/2)\cdot u dx\right)_t+C(\|\nabla H\|_{L^2}^2+\|\nabla u\|_{L^2}^2) \nonumber \\
      &+C\sigma^m(\|H_t\|_{L^2}^{2}+\|\curl^2  H\|_{L^2}^2)+C\sigma^m \|\nabla u\|_{L^3}^3+C\sigma^m\|\nabla H\|_{L^2}^2\|\nabla u\|_{L^2}^2 \nonumber \\
      &+ C\sigma^m\|\nabla H\|_{L^2}^2(\|\nabla H\|_{L^2}^4+\|\nabla u\|_{L^2}^4),
\end{align}
Making use of the results \eqref{I1}, \eqref{I2},\eqref{I3} and \eqref{I4}, it follows from $\eqref{I0}$ that
\begin{align}\label{I01}
&\left((\lambda+2\mu)\int\sigma^{m}(\div u)^{2}dx+\mu\int\sigma^{m}|\omega|^{2}dx\right)_{t}+\int\sigma^{m}\rho|\dot{u}|^{2}dx \nonumber \\
\leq &\left(\int\sigma^{m}(P-{P}_{\infty})\,\div udx\right)_{t}+\left(\int\sigma^m (H \cdot \nabla H- \nabla|H|^2/2)\cdot u dx\right)_t \nonumber \\
     &+Cm\sigma^{m-1}\sigma'C_0 +C(\|\nabla H\|_{L^2}^2+\|\nabla u\|_{L^2}^2)(\|\nabla u\|_{L^{2}}^2+1)+C\sigma^{m}\|\nabla u\|_{L^{3}}^{3} \nonumber \\
     &+C\sigma^m(\|H_t\|_{L^2}^{2}+\|\curl^2  H\|_{L^2}^2)+ C\sigma^m\|\nabla H\|_{L^2}^2(\|\nabla H\|_{L^2}^4+\|\nabla u\|_{L^2}^4),
\end{align}
integrating over $(0,T]$, by \eqref{tdu1}, Lemma \ref{lem-basic} and Young's inequality, we conclude that for any $m>0$,
\begin{align}\label{I02}
\displaystyle  &\sigma^{m}\|\nabla u\|_{L^{2}}^{2}+\int_0^T\int\sigma^{m}\rho|\dot{u}|^{2}dxdt \nonumber \\
\leq & C\int_0^T \sigma^m(\|H_t\|_{L^2}^{2}+\|\curl^2  H\|_{L^2}^2) dt
+C\int_0^T\sigma^{m}\|\nabla H\|_{L^2}^2(\|\nabla H\|_{L^2}^4+\|\nabla u\|_{L^2}^4)dt \nonumber \\
     &+CC_{0}+C\int_0^T\sigma^{m}\|\nabla u\|_{L^{2}}^2(\|\nabla u\|_{L^{2}}^2+\|\nabla H\|_{L^{2}}^2)dt+C\int_0^T\sigma^{m}\|\nabla u\|_{L^{3}}^{3}dt.
\end{align}
Choose $m=1,$ together with \eqref{key1}, \eqref{basic3} and \eqref{tdh}, we obtain $\eqref{a01}$. The proof of Lemma \ref{lem-tdh} is completed.
\end{proof}

\begin{lemma}\label{lem-a0}
Under the conditions of Proposition \ref{pr1}. Then there is a positive constant
  $\ve_2 $ depending only  on $\mu$, $\lambda$, $\nu$, $a$, $\gamma$, ${\rho}_{\infty}$ and $\bar{\rho}$ such that if $C_0\leq \ve_2$,
 \begin{align}\label{a02}
 \displaystyle  A_2(T)
    \le   C C_0^{\frac{31}{54}} + CA_1(T)  + C\int_0^{T}\sigma^2 \|\nabla u\|_{L^4}^4 dt.
 \end{align}
\end{lemma}
\begin{proof} Operating $ \sigma^{m}\dot{u}^{j}[\pa/\pa t+\div (u\cdot)] $ to $ (\ref{2m2})^j,$ summing with respect to $j$, and integrating over $\Omega,$ together with $ \eqref{CMHD}_1 $, we get
\begin{align}\label{J0}
\displaystyle  &\left(\frac{\sigma^{m}}{2}\int\rho|\dot{u}|^{2}dx\right)_t-\frac{m}{2}\sigma^{m-1}\sigma'\int\rho|\dot{u}|^{2}dx \nonumber \\
 =& \int\sigma^{m}(\dot{u}\cdot\nabla F_t+\dot{u}^{j}\,\div(u\partial_jF))dx \nonumber \\
&\quad+\mu\int\sigma^{m}(-\dot{u}\cdot\nabla\times\omega_t-\dot{u}^{j}\div((\nabla\times\omega)^j\,u))dx \nonumber \\
&\quad+\int\sigma^{m}(\dot{u}\cdot(\div (H\otimes H))_t+\dot{u}^{j}\div((\div (H\otimes H^j)\,u))dx \nonumber \\
 \triangleq & J_1+ J_2+ J_3.
\end{align}
Let us estimate $J_1, J_2$ and $J_3$.
By \eqref{navier-b} and \eqref{CMHD1}$_1$, a direct computation yields
\begin{align}\label{J10}
J_1 
& = \int_{\partial\Omega}\sigma^{m}F_t\dot{u}\cdot nds-\int\sigma^{m}F_t\,\div\dot{u}dx-\int\sigma^{m}u \cdot \nabla\dot{u}^j\partial_jFdx \nonumber \\
& = \int_{\partial\Omega}\sigma^{m}F_t\dot{u}\cdot nds - (\lambda+2\mu)\int\sigma^{m}(\div\dot{u})^{2}dx + (\lambda+2\mu)\int\sigma^{m}\div\dot{u}\,\nabla u:\nabla udx \nonumber\\
& \quad -\gamma\int\sigma^{m} P\div\dot{u}\,\div udx+\int\sigma^{m}\div \dot{u}\, u \cdot \nabla F dx-\int\sigma^{m}u \cdot\nabla\dot{u}^j\partial_jF dx \nonumber \\
 &\quad +\int\sigma^{m}\div\dot{u}\,H \cdot H_tdx+\int\sigma^{m}\div\dot{u}\,u \cdot \nabla H \cdot H dx \nonumber\\
& \leq \int_{\partial\Omega}\sigma^{m}F_t\dot{u}\cdot nds - (\lambda+2\mu)\int\sigma^{m}(\div\dot{u})^{2}dx + \frac{\delta}{12}\sigma^{m}\|\nabla\dot{u}\|\ltwo+C\sigma^m \|\nabla u\|^4_{L^4}\nonumber\\
& \quad+C\sigma^m \big(\|\nabla u\|_{L^2}^2 \|\nabla F\|_{L^3}^2+ C_0^{\frac{2}{27}}\|\nabla H_t\|_{L^2}^2+\|\nabla u\|_{L^2}^2\|\nabla H\|_{L^2}^2\|\nabla H\|_{L^6}^2\big),
\end{align}
where in the second equality we have used
\begin{align*}
\displaystyle F_t=(2\mu+\lm)\div\dot  u-(2\mu+\lm)\na u:\na u  - u\cdot\na F+u \cdot \nabla H \cdot H+\ga P\div u-H \cdot H_t.
\end{align*}
For the boundary term on the righthand side of \eqref{J10}, using Lemma \ref{lem-be} with $f=F$, we have
\begin{align}\label{J11}
&\int_{\partial\Omega}\sigma^{m}F_t\dot{u}\cdot nds=-\int_{\partial\Omega}\sigma^{m}F_t\,(u\cdot\nabla n\cdot u)ds \nonumber\\
= & -\left(\int_{\partial\Omega}\sigma^{m}(u\cdot\nabla n\cdot u)Fds\right)_t+m\sigma^{m-1}\sigma'\int_{\partial\Omega}(u\cdot\nabla n\cdot u)Fds \nonumber\\
&\quad +\int_{\partial\Omega}\sigma^{m}\big(F\dot{u}\cdot\nabla n \cdot u+Fu\cdot\nabla n \cdot\dot{u}\big)ds \nonumber \\
&\quad  -\int_{\partial\Omega}\sigma^{m}\big(F(u \cdot \nabla) u\cdot\nabla n \cdot u+Fu\cdot\nabla n \cdot(u \cdot \nabla) u\big)ds\nonumber\\
\leq& -\left(\int_{\partial\Omega}\sigma^{m}(u\cdot\nabla n\cdot u)Fds\right)_t+Cm\sigma^{m-1}\sigma'\|\nabla u\|_{L^2}^{2}\|\nabla F\|_{L^2}+\frac{\delta}{12}\sigma^{m}\|\nabla\dot{u}\|_{L^2}^2 \nonumber\\
&\quad+C\sigma^{m}(\|\nabla u\|_{L^2}^{2}\|\nabla F\|_{L^2}^{2}+\|\nabla u\|_{L^2}^{3}\|\nabla F\|_{L^2}+\|\na F\|_{L^6}\|\na u\|^3_{L^2}+\|\na u\|_{L^4}^4).
\end{align}
From Lemma \ref{lem-f-td} and \eqref{h-tdh}, we have
\begin{align}
&\|\nabla u\|_{L^2}^{2}\|\nabla F\|_{L^2}^{2}\leq C(\|\sqrt{\rho}\dot{u}\|_{L^2}^{2}+\|\curl^2  H\|_{L^2}^{2})\|\nabla u\|_{L^2}^{2}+C\|\nabla u\|_{L^2}^{2}\|\nabla H\|_{L^2}^{2}, \\
&\|\na F\|_{L^6}\|\na u\|^3_{L^2}\leq \frac{\delta}{12}\sigma^{m}\|\nabla\dot{u}\|_{L^2}^2+C\|\na u\|^6_{L^2}+C\|\curl^2  H\|_{L^2}^{2}\|\na u\|^4_{L^2}\nonumber \\
 &\quad\qquad\quad\quad\quad\quad+C\|\nabla H\|_{L^2}^{2}+C\|\nabla H\|_{L^2}^{4},\\
 &\|\nabla u\|_{L^2}^{2}\|\nabla F\|_{L^3}^{2}\leq \frac{\delta}{12}\sigma^{m}\|\nabla\dot{u}\|_{L^2}^2+C\|\sqrt{\rho}\dot{u}\|_{L^2}^{2}\|\nabla u\|_{L^2}^{4} \nonumber \\
 &\quad\qquad\quad\quad\quad\quad +C\|\nabla u\|_{L^2}^{2}\|\nabla H\|_{L^2}^{2}\|\curl^2  H\|_{L^2}^{2}+C\|\nabla u\|_{L^2}^{2}\|\nabla H\|_{L^2}^{4}.\label{j11-1}
\end{align}
Putting \eqref{J11}-\eqref{j11-1} into \eqref{J10}, we have
\begin{align}\label{J1}
& J_1 \leq - (\lambda+2\mu)\int\sigma^{m}(\div\dot{u})^{2}dx -\left(\int_{\partial\Omega}\sigma^{m}(u\cdot\nabla n\cdot u)Fds\right)_t\nonumber\\
& \quad + \frac{\delta}{3}\sigma^{m}\|\nabla\dot{u}\|\ltwo+C\sigma^m C_0^{\frac{2}{27}}\|\nabla H_t\|\ltwo+C\sigma^m(\|\nabla H\|_{L^2}^4+\|\nabla u\|_{L^2}^4)\|\curl^2  H\|\ltwo \nonumber\\
& \quad +C\sigma^{m}(\|\sqrt{\rho}\dot{u}\|_{L^2}^{2}+\|\curl^2  H\|\ltwo)\|\nabla u\|_{L^2}^{2}+C\sigma^{m}\|\sqrt{\rho}\dot{u}\|\ltwo \|\nabla u\|_{L^2}^4\nonumber\\
&\quad +C\sigma^m \|\nabla u\|^4_{L^4}+C\sigma^m(\|\nabla u\|_{L^2}^2+1)(\|\nabla u\|_{L^2}^4+\|\nabla H\|_{L^2}^4)+C\sigma^m\|\nabla H\|_{L^2}^2\nonumber\\
&\quad+Cm\sigma^{m-1}\sigma'(\|\sqrt{\rho}\dot{u}\|_{L^2}^{2}+\|\curl^2  H\|\ltwo+\|\nabla u\|_{L^2}^{4}+\|\nabla H\|_{L^2}^{2}).
\end{align}

Next, noticing $ \omega_t=\curl \dot u-u\cdot \na \omega-\na u^i\times \pa_iu$ and \eqref{navier-b}, it follows
\begin{align}\label{J20}
J_2 
&=-\mu\int\sigma^{m}|\curl \dot{u}|^{2}dx+\mu\int\sigma^{m}\curl\dot{u}\cdot(\nabla u^i\times\nabla_i u)dx \nonumber\\
&\quad-\mu\int\sigma^{m}\div u (\omega \cdot \curl\dot{u})dx-\mu\int\sigma^{m} (\omega \times\nabla u^i)  \cdot\nabla_i\dot{u}  dx\nonumber\\
&\leq -\mu\int\sigma^{m}|\curl \dot{u}|^{2}dx+\frac{\delta}{3}\sigma^{m}\|\nabla\dot{u}\|_{L^2}^2+C\sigma^{m}\|\nabla u\|_{L^4}^4.
\end{align}

Finally, a directly computation shows that
\begin{align}\label{J30}
\displaystyle J_3 
& =-\int\sigma^{m}\nabla\dot{u}:(H\otimes H)_t dx-\mu\int\sigma^{m}H \cdot \nabla H^j u \cdot \nabla\dot{u}^{j}dx\nonumber \\
& \leq C\sigma^{m}(\|\nabla\dot{u}\|_{L^2} \|H\|_{L^3} \|H_t\|_{L^6}+\|\nabla\dot{u}\|_{L^2} \|H\|_{L^6} \|\nabla H\|_{L^6}\|u\|_{L^6} ) \nonumber \\
& \leq \frac{\delta}{3}\sigma^{m}\|\nabla\dot{u}\|_{L^2}^2+C\sigma^m(\|\nabla H\|_{L^2}^4+\|\nabla u\|_{L^2}^4)\|\curl^2  H\|\ltwo \nonumber\\
& \quad +C\sigma^m C_0^{\frac{2}{27}}\|\nabla H_t\|\ltwo+C\sigma^m \|\nabla H\|_{L^2}^4\|\nabla u\|_{L^2}^2.
\end{align}

Combining \eqref{J1}, \eqref{J20} with \eqref{J30}, we deduce from \eqref{J0} that
\begin{align}\label{J01}
&\left(\frac{\sigma^{m}}{2}\|\sqrt{\rho}\dot{u}\|_{L^2}^2\right)_t+(\lambda+2\mu)\sigma^{m}\|\div\dot{u}\|_{L^2}^2+\mu\sigma^{m}\|\curl\dot{u}\|_{L^2}^2 \nonumber\\
&\leq -\left(\int_{\partial\Omega}\sigma^{m}(u\cdot\nabla n\cdot u)Fds\right)_t+
\delta\sigma^{m}\|\nabla\dot{u}\|\ltwo+C\sigma^m C_0^{\frac{2}{27}}\|\nabla H_t\|\ltwo \nonumber\\
& \quad +C\sigma^m \|\nabla u\|^4_{L^4}+C\sigma^m(\|\nabla H\|_{L^2}^4+\|\nabla u\|_{L^2}^4)\|\curl^2  H\|\ltwo \nonumber\\
& \quad +C\sigma^{m}(\|\sqrt{\rho}\dot{u}\|_{L^2}^{2}+\|\curl^2  H\|\ltwo)\|\nabla u\|_{L^2}^{2}+C\sigma^{m}\|\sqrt{\rho}\dot{u}\|\ltwo \|\nabla u\|_{L^2}^4\nonumber\\
&\quad +C\sigma^m(\|\nabla u\|_{L^2}^2+1)(\|\nabla u\|_{L^2}^4+\|\nabla H\|_{L^2}^4)+C\sigma^m\|\nabla H\|_{L^2}^2\nonumber\\
&\quad+Cm\sigma^{m-1}\sigma'(\|\sqrt{\rho}\dot{u}\|_{L^2}^{2}+\|\curl^2  H\|\ltwo+\|\nabla u\|_{L^2}^{4}+\|\nabla H\|_{L^2}^{2}).
\end{align}
As observed in \cite{cl2019}, it follows from \eqref{bdd3} that
\begin{equation}
\displaystyle  (\dot{u}-(u\cdot\nabla n)\times u^{\perp})\cdot n=0,
\end{equation}
which together with Lemma \ref{lem-vn} yields
\begin{align}
\|\nabla\dot{u}\|_{L^2}&\leq C(\|\div \dot{u}\|_{L^2}+\|\curl \dot{u}\|_{L^2}+\|\nabla[(u\cdot\nabla n)\times u^\perp]\|_{L^2})\nonumber \\
 & \leq C(\|\div \dot{u}\|_{L^2}+\|\curl \dot{u}\|_{L^2}+\|\nabla u\|_{L^2}^2+\|\nabla u\|_{L^4}^2).\label{tdudot}
\end{align}
By \eqref{tdudot} and Lemma \ref{lem-basic}, choosing $\delta$ small enough, and integrating \eqref{J01} over $(0,T]$, for $m>0$, we get
\begin{align}\label{J03}
&\quad \sigma^{m}\|\sqrt{\rho}\dot{u}\|_{L^2}^2+\int_0^T\sigma^{m}\|\nabla\dot{u}\|_{L^2}^2dt\nonumber\\
&\leq -\int_{\partial\Omega}\sigma^{m}(u\cdot\nabla n\cdot u)Fds+CC_0^{\frac{2}{27}}\int_0^T\sigma^m \|\nabla H_t\|\ltwo dt \nonumber\\
& \quad +C\int_0^T\sigma^m \|\nabla u\|^4_{L^4}dt+CC_0^{\frac{2}{9}}\sup_{0\leq t\leq T}\sigma^m(\|\curl^2  H\|\ltwo+\|\sqrt{\rho}\dot{u}\|_{L^2}^{2}) \nonumber\\
& \quad +CC_0^{\frac{2}{9}}\sup_{0\leq t\leq T}\sigma^{m}(\|\nabla H\|\ltwo+\|\nabla u\|_{L^2}^{2})+CC_0\sup_{0\leq t\leq \sigma(T)}\sigma^{m-1}\|\nabla u\|_{L^2}^{2} \nonumber\\
&\quad+C \int_0^{\sigma(T)} m\sigma^{m-1}(\|\sqrt{\rho}\dot{u}\|_{L^2}^{2}+\|\curl^2  H\|\ltwo)dt +CC_0.
\end{align}
For the boundary term in the right-hand side of \eqref{J03} , using Lemma \ref{lem-be} again, we have
\begin{align}\label{J0b1}
&\quad\int_{\partial\Omega}(u\cdot\nabla n\cdot u)Fds\leq C\|\nabla u\|_{L^2}^{2}\|\nabla F\|_{L^2}\nonumber\\
&\leq\frac{1}{2}\|\sqrt{\rho}\dot{u}\|_{L^2}^2+CC_0^{\frac{2}{27}}\|\curl^2  H\|\ltwo+C(\|\nabla H\|_{L^2}^2+\|\nabla u\|_{L^2}^4).
\end{align}
Therefore,
\begin{align}\label{J04}
&\sigma^{m}\|\sqrt{\rho}\dot{u}\|_{L^2}^2+\int_0^T\sigma^{m}\|\nabla\dot{u}\|_{L^2}^2dt-C_2C_0^{\frac{2}{27}}\int_0^T\sigma^m \|\nabla H_t\|\ltwo dt \nonumber\\
&\leq C\int_0^T\sigma^m \|\nabla u\|^4_{L^4}dt+CC_0^{\frac{2}{9}}\sup_{0\leq t\leq T}\sigma^m(\|\curl^2  H\|\ltwo+\|\sqrt{\rho}\dot{u}\|_{L^2}^{2}) \nonumber\\
& \quad +CC_0^{\frac{2}{9}}\sup_{0\leq t\leq T}\sigma^{m}(\|\nabla H\|\ltwo+\|\nabla u\|_{L^2}^{2})+CC_0\sup_{0\leq t\leq \sigma(T)}\sigma^{m-1}\|\nabla u\|_{L^2}^{2} \nonumber\\
&\quad+C \int_0^{\sigma(T)} m\sigma^{m-1}(\|\sqrt{\rho}\dot{u}\|_{L^2}^{2}+\|\curl^2  H\|\ltwo)dt +CC_0 \nonumber\\
&\quad+CC_0^{\frac{2}{27}}\sigma^{m}\|\curl^2  H\|\ltwo+C\sigma^{m}(\|\nabla H\|_{L^2}^2+\|\nabla u\|_{L^2}^4).
\end{align}

Next, we need to estimate the term $\|\nabla H_t\|_{L^2}$. Noticing that
\begin{equation}\label{ht}
\begin{cases}
 H_{tt}-\nu \nabla \times (\curl H_t)=(H \cdot \nabla u-u \cdot \nabla H-H \div u)_t,&\text{in}\quad \Omega,\\
 H_t \cdot n=0,\quad \curl H_t \times n=0,& \text{on}\quad \partial\Omega,\\
 \end{cases}
\end{equation}
and after directly computations we obtain
\begin{align}\label{ht1}
&\quad \left(\frac{\sigma^{m}}{2}\|H_t\|_{L^2}^2\right)_t+\sigma^{m}\|\curl H_t\|_{L^2}^2-\frac{m}{2}\sigma^{m-1}\sigma' \|H_t\|\ltwo\nonumber\\
&= \int \sigma^{m}(H_t \cdot \nabla u-u \cdot \nabla H_t-H_t \div u)\cdot H_t dx \nonumber \\
&\quad+ \int \sigma^{m}(H \cdot \nabla \dot{u}-\dot{u} \cdot \nabla H-H \div \dot{u})\cdot H_t dx \nonumber \\
&\quad - \int \sigma^{m}(H \cdot \nabla (u \cdot \nabla u)-(u \cdot \nabla u)\cdot \nabla H-H \div(u \cdot \nabla u) )\cdot H_t dx \nonumber \\
& \triangleq K_1+K_2+K_3.
\end{align}
By Lemma \ref{lem-gn} and Lemma \ref{lem-f-td}, a direct calculation leads to
\begin{align}\label{htk1}
K_1 
& \leq C\sigma^{m}(\|H_t\|_{L^3}\|H_t\|_{L^6}\|\nabla u\|_{L^2}
+\|u\|_{L^6}\|H_t\|_{L^3}\|\nabla H_t\|_{L^2})\nonumber\\
&\leq \frac{\delta}{4}\sigma^{m}\|\nabla H_t\|_{L^2}^{2}+C\sigma^{m}\|\nabla u\|_{L^2}^4\|H_t\|_{L^2}^{2}.
\end{align}
Similarly,
\begin{align}\label{htk20}
K_2 
& \leq C\sigma^{m}\|H\|_{L^3}\|H_t\|_{L^6}\|\nabla \dot{u}\|_{L^2}
-\int_{\partial \Omega}\sigma^{m}(\dot{u} \cdot n)( H \cdot H_t) ds \nonumber\\
 & \quad +\int\sigma^{m}\div\dot{u}\, H \cdot H_t dx+\int\sigma^{m} \dot{u}\cdot \nabla H_t \cdot H dx \nonumber\\
&\leq \int_{\partial \Omega}\sigma^{m}(u \!\cdot \!\nabla n \cdot u)( H\! \cdot\! H_t) ds+CC_0^{\frac{1}{27}}\sigma^{m}(\|\nabla \dot{u}\|_{L^2}^{2}\!+\!\|\nabla H_t\|_{L^2}^{2}).
\end{align}
For the boundary term in the last inequality, we use the similar method as that used in Lemma \ref{lem-be} to get that
\begin{align}\label{htk21}
& \quad\int_{\partial \Omega}\sigma^{m}(u \cdot \nabla n \cdot u)( H \cdot H_t) ds\nonumber\\
&\leq C\sigma^{m}(\|u\|_{L^6}\|\nabla u\|_{L^2}\|H\|_{L^6}\|H_t\|_{L^6}+\|u\|_{L^6}^2\|\nabla H\|_{L^2}\|H_t\|_{L^6}\nonumber\\
& \quad+\|u\|_{L^6}^2\|\nabla H_t\|_{L^2}\|H\|_{L^6}+\|u\|_{L^6}^{2}\|H\|_{L^6}\|H_t\|_{L^6})\nonumber\\
&\leq \frac{\delta}{4}\sigma^{m}\|\nabla H_t\|_{L^2}^{2}+C\sigma^{m}\|\nabla u\|_{L^2}^4\|\nabla H\|_{L^2}^{2}.
\end{align}
Combining \eqref{htk20} and \eqref{htk21}, we have
\begin{align}\label{htk2}
K_2 &\leq  \frac{\delta}{4}\sigma^{m}\|\nabla H_t\|_{L^2}^{2}+C\sigma^{m}(\|\nabla u\|_{L^2}^4\|\nabla H\|_{L^2}^{2}
+C^{\frac{1}{27}}(\|\nabla \dot{u}\|_{L^2}^{2}+\|\nabla H_t\|_{L^2}^{2})).
\end{align}
Similarly, by \eqref{htk21}, a direct computation yields
\begin{align}\label{htk3}
K_3 
&\leq \frac{\delta}{2}\sigma^{m}\|\nabla H_t\|_{L^2}^{2}+C\sigma^{m}(\|\nabla u\|_{L^2}^4+\|\nabla H\|_{L^2}^{4})(\|\sqrt{\rho}\dot{u}\|_{L^2}^{2}+\|\curl^2  H\|\ltwo)\nonumber \\
 &\quad+C\sigma^{m}\|\nabla u\|_{L^2}^2\|\nabla H\|_{L^2}^{2}(\|\nabla u\|_{L^2}^2+\|\nabla H\|_{L^2}^{2}+1).
\end{align}
Putting \eqref{htk1}, \eqref{htk2} and \eqref{htk3} into \eqref{ht1}, choosing $\delta$ small enough, we have
\begin{align}\label{ht3}
&\quad \left(\sigma^{m}\|H_t\|_{L^2}^2\right)_t+\sigma^{m}\|\nabla H_t\|_{L^2}^2-CC_0^{\frac{1}{27}}\sigma^{m}(\|\nabla \dot{u}\|_{L^2}^{2}+\|\nabla H_t\|_{L^2}^{2})\nonumber\\
&\leq C\sigma^{m}(\|\nabla u\|_{L^2}^4+\|\nabla H\|_{L^2}^{4})(\|\sqrt{\rho}\dot{u}\|_{L^2}^{2}+\|\curl^2  H\|\ltwo+\|H_t\|_{L^2}^2)\nonumber \\
 &\quad+C\sigma^{m}\|\nabla u\|_{L^2}^2\|\nabla H\|_{L^2}^{2}(\|\nabla u\|_{L^2}^2+\|\nabla H\|_{L^2}^{2}+1)+Cm\sigma^{m-1}\sigma' \|H_t\|\ltwo.
\end{align}
Integrating over $(0,T]$, then by Lemma \ref{lem-basic}, for $m>0$, we get
\begin{align}\label{ht4}
&\quad \sigma^{m}\|H_t\|_{L^2}^2+\int_0^T\sigma^{m}\|\nabla H_t\|_{L^2}^2dt-C_3C^{\frac{1}{27}}\int_0^T\sigma^{m}(\|\nabla \dot{u}\|_{L^2}^{2}+\|\nabla H_t\|_{L^2}^{2})dt\nonumber\\
&\leq CC_0^{\frac{2}{9}}\sup_{0\leq t\leq T}\sigma^{m}(\|\sqrt{\rho}\dot{u}\|_{L^2}^{2}+\|\curl^2  H\|\ltwo+\|H_t\|_{L^2}^2)+CC_0\sup_{0\leq t\leq T}\sigma^{m}\|\nabla u\|_{L^2}^2\nonumber \\
 &\quad+CC_0^{\frac{2}{9}}\sup_{0\leq t\leq T}\sigma^{m}(\|\nabla u\|_{L^2}^2+\|\nabla H\|_{L^2}^{2})+C \int_0^{\sigma(T)} m\sigma^{m-1}\|H_t\|\ltwo dt.
\end{align}
Now take $m=2$ in \eqref{J04} and \eqref{ht4}, we deduce after adding them together that
\begin{align}\label{a20}
&\sigma^{2}(\|\sqrt{\rho}\dot{u}\|_{L^2}^2+\|H_t\|_{L^2}^2)+\int_0^T\sigma^{2}(\|\nabla\dot{u}\|_{L^2}^2+\|\nabla H_t\|_{L^2}^2)dt \nonumber \\
& \quad -C_2C_0^{\frac{2}{27}}\int_0^T\sigma^2 \|\nabla H_t\|\ltwo dt -C_3C^{\frac{1}{27}}\int_0^T\sigma^{2}(\|\nabla \dot{u}\|_{L^2}^{2}+\|\nabla H_t\|_{L^2}^{2})dt\nonumber\\
&\leq C\int_0^T\sigma^2 \|\nabla u\|^4_{L^4}dt+CC_0^{\frac{2}{9}}\sup_{0\leq t\leq T}\sigma^2(\|\sqrt{\rho}\dot{u}\|_{L^2}^{2}+\|\curl^2  H\|\ltwo+\|H_t\|_{L^2}^2) \nonumber\\
& \quad +CC_0^{\frac{2}{9}}\sup_{0\leq t\leq T}\sigma^{2}(\|\nabla H\|\ltwo+\|\nabla u\|_{L^2}^{2})+CC_0\sup_{0\leq t\leq T}\sigma\|\nabla u\|_{L^2}^{2} \nonumber\\
&\quad+C \int_0^{\sigma(T)} \sigma(\|\sqrt{\rho}\dot{u}\|_{L^2}^{2}+\|\curl^2  H\|\ltwo+\|H_t\|\ltwo)dt +CC_0 \nonumber\\
&\quad+CC_0^{\frac{2}{27}}\sigma^{2}\|\curl^2  H\|\ltwo+C\sigma^{2}(\|\nabla H\|_{L^2}^2+\|\nabla u\|_{L^2}^4)\nonumber\\
&\leq C\int_0^T\sigma^2 \|\nabla u\|^4_{L^4}dt+CC_0^{\frac{13}{18}}+CA_1(T)+CC_0+CC_0^{\frac{31}{54}}.
\end{align}
Thus we have
\begin{align}\label{a21}
&\quad \sup_{0\leq t\leq T}\sigma^{2}(\|\sqrt{\rho}\dot{u}\|_{L^2}^2+\|H_t\|_{L^2}^2)+\int_0^T\sigma^{2}(\|\nabla\dot{u}\|_{L^2}^2+\|\nabla H_t\|_{L^2}^2)dt \nonumber \\
&\leq C\int_0^T\sigma^2 \|\nabla u\|^4_{L^4}dt+CA_1(T)+CC_0^{\frac{31}{54}}.
\end{align}
provided that $C_0$ is chosen to satisfy
\begin{align*}
	C_0\leq \ve_2 \triangleq \min \{\ve_1, (4C_2)^{-{\frac{27}{2}}},(4C_3)^{-27}\}.
\end{align*}
Finally, by Lemma \ref{lem-gn} and \eqref{CMHD}$_3$, it holds
\begin{align}\label{h2xd1}
\|\curl^2  H\|_{L^2}
&\leq C(\|H_t\|_{L^2}+\|\curl^2  H\|_{L^2}^{\frac{1}{2}}\|\nabla H\|_{L^2}^{\frac{1}{2}}\|\nabla u\|_{L^2}+\|\nabla H\|_{L^2}\|\nabla u\|_{L^2})\nonumber \\
 &\leq \frac{1}{2}\|\curl^2  H\|_{L^2}\!+\!C(\|H_t\|_{L^2}\!+\!\|\nabla H\|_{L^2}\|\nabla u\|_{L^2}^2\!+\!\|\nabla H\|_{L^2}\|\nabla u\|_{L^2}).
\end{align}
Thus, by \eqref{key1} and \eqref{h2xd1}, we have
\begin{align}\label{h2xd2}
& \quad \sup_{0\leq t\leq T}\sigma^{2} \|\curl^2  H\|_{L^2}^2
\leq  C\int_0^T\sigma^2 \|\nabla u\|^4_{L^4}dt+CA_1(T)+CC_0^{\frac{31}{54}}.
\end{align}
Combining \eqref{a21} and \eqref{h2xd2}, we give \eqref{a02} and complete the proof of Lemma \ref{lem-a0}.
\end{proof}

\begin{lemma}\label{lem-a1}
Under the conditions of Proposition \ref{pr1}. Then there exist  positive constants $\tilde{C}=C(\bar{\rho},M_1,M_2)$ and $\varepsilon_3$ depending only on $\mu,\,\,\lambda,\,\,\nu,\,\,\gamma,\,\,a,\,\,{\rho}_{\infty},\,\,\bar{\rho}$, $M_1$  and $M_2$ such that if $C_0<\varepsilon_3$,
\begin{align}
& \sup_{0\le t\le \sigma(T)}  \|\nabla u\|_{L^2}^2+\int_0^{\si(T)} \|\sqrt{\rho}\dot{u}\|_{L^2}^2dt\le \tilde{C},\label{uv1}\\
&\sup_{0\le t\le \si(T)}\!\!\! t (\|\sqrt{\rho}\dot{u}\|_{L^2}^2\!+\!\|\curl^2 H\|_{L^2}^2\!\!+\!\!\|H_t\|_{L^2}^2)
\!+\!\!\int_0^{\si(T)}\!\!\!\! t (\|\nabla\dot{u}\|_{L^2}^2\!+\!\|\nabla H_t\|_{L^2}^2)dt\!\leq \tilde{C}.\label{uv2}
\end{align}
\end{lemma}

\begin{proof} As we have done in the proof of Lemma \ref{lem-f-td}, multiplying \eqref{CMHD1}$_2$ by $u_{t}$ and integrating over $\Omega$, using \eqref{key1},  Sobolev's and Young's inequalities leads to
\begin{align}\label{bw2}
&\quad \left(\frac{\lambda+2\mu}{2}\int(\div  u )^{2}dx+\frac{\mu}{2}\int|\curl  u |^{2}dx-\int(P-{P}_{\infty}-\frac{|H|^2}{2})\div  u dx\right)_t+\int\rho|\dot{u}|^{2}dx \nonumber\\
& =\left(\int(H \cdot\nabla H)\cdot udx\right)_t+\int\rho\dot{u}\cdot(u\cdot\nabla  u )dx-\int P_t\div  u dx \nonumber\\
&\quad-\int (H \cdot\nabla H- \nabla H \cdot H)_t \cdot udx \nonumber\\
 & \triangleq \frac{d}{dt}L_0+L_1+L_2+L_3.
\end{align}
By \eqref{g1} and \eqref{h-l3}, we have
\begin{align}\label{bw2-0}
L_0\leq C\|H\|_{L^3}\|\nabla H\|_{L^2}\|u\|_{L^6}\leq \frac{\mu}{4}\|\nabla u\|_{L^2}^2+C C_0^{\frac{2}{27}}\|\nabla H\|_{L^2}^2.
\end{align}
Using Lemma \ref{lem-gn} and \eqref{key1} yields
\begin{align}\label{bw2-1}
L_1 &=\int\rho\dot{u}\cdot(u\cdot\nabla u)dx \nonumber\\
 &\leq C\|\rho^{\frac{1}{2}}\dot{u}\|_{L^2}\|\rho^{1/3}u\|_{L^3}\|\nabla u\|_{L^6}\nonumber \\
 &\leq C_4C^{\frac{1}{27}}\|\rho^{\frac{1}{2}}\dot{u}\|_{L^2}^2+C(\|\nabla u\|_{L^2}^2\!+\!\|P\!-\!{P}_{\infty}\|_{L^6}^2\!+\!\|P\!-\!{P}_{\infty}\|_{L^2}^{2}).
\end{align}
Next, by \eqref{CMHD1}$_1$, \eqref{flux}, \eqref{key1}, Lemma \ref{lem-f-td}, Sobolev’s and Young’s inequalities leads to
\begin{align}\label{bw2-2}
&L_2=- \frac{1}{\lambda+2\mu}\int(P-{P}_{\infty})(F \div u+\nabla F \cdot u ) dx  \nonumber \\
& \quad - \frac{1}{2(\lambda+2\mu)}\int(P-{P}_{\infty})^{2}\div u dx +\gamma\int P(\div u)^2dx \nonumber\\
& \leq C(\|\nabla u\|_{L^2}\|F \|_{L^2}+\|P-{P}_{\infty}\|_{L^3}\|\nabla F \|_{L^2}\|u\|_{L^6})
+\|P-{P}_{\infty}\|_{L^2}\|\nabla u\|_{L^2}+\|\nabla u\|_{L^2}^2)\nonumber\\
&\leq C\|\nabla u\|_{L^2}(\|\sqrt{\rho}\dot{u}\|_{L^2}+\|\nabla  u \|_{L^2}+\|P-{P}_{\infty}\|_{L^2}+C_0^{\frac{1}{27}}(\|\curl^2  H\|_{L^2}+\|\nabla H\|_{L^2}))\nonumber\\ 
&\leq \frac{1}{4}\|\sqrt{\rho}\dot{u}\|_{L^2}^{2}
\!+\!C(\|\nabla  u \|_{L^2}^2\!+\!\|P\!-\!{P}_{\infty}\|_{L^2}^{2}+C_0^{\frac{2}{27}}(\|\curl^2  H\|_{L^2}^2+\|\nabla H\|_{L^2}^2)).
\end{align}
Using Lemma \ref{lem-gn} and \eqref{tdh}, a directly computation yields
\begin{align}\label{bw2-3}
L_3 &=-\int (H_t \cdot\nabla H- \nabla H \cdot H_t) \cdot u dx-\int (H \cdot\nabla H_t- \nabla H_t \cdot H) \cdot u dx \nonumber\\
 &\leq C(\|H_t\|_{L^2}\|\nabla H\|_{L^3}\|\nabla u \|_{L^2}+\|H_t\|_{L^2}\|H\|_{L^3}\|\nabla u \|_{L^6}) \nonumber \\
 &\leq C_5C_0^{\frac{1}{27}}(\|\rho^{\frac{1}{2}}\dot{u }\|_{L^2}^2+\|\nabla u \|_{L^2}^2+\|H_t\|_{L^2}^2+\|\curl^2  H\|_{L^2}^2+\|\nabla H\|_{L^2}^2+\|P\!-\!{P}_{\infty}\|_{L^6}^{2})\nonumber \\
 &\quad +C\|H_t\|_{L^2}^2+C(\|\nabla H\|_{L^2}^2+\|\curl^2  H\|_{L^2}\|\nabla H\|_{L^2})\|\nabla u \|_{L^2}^2.
\end{align}
Putting \eqref{bw2-0}-\eqref{bw2-3} into \eqref{bw2}, 
we obtain
\begin{align}\label{bw22}
\displaystyle  & \quad\left((\lambda+2\mu)\|\div  u \|_{L^{2}}^{2}+\mu\|\curl  u \|_{L^{2}}^{2}-2\int(P-{P}_{\infty}-\frac{|H|^2}{2})\div  u dx\right)_t
+\int\rho|\dot{u}|^{2}dx \nonumber \\
 & \le (\frac{\mu}{4}\|\nabla u\|_{L^2}^2+C C_0^{\frac{2}{27}}\|\nabla H\|_{L^2}^2)_t+C(1+\|\nabla H\|_{L^2}^2+\|\curl^2  H\|_{L^2}\|\nabla H\|_{L^2})\|\nabla u \|_{L^2}^2\nonumber\\&\quad+C\left(\!\|P\!-\!{P}_{\infty}\|_{L^6}^2\!+\!\|P\!-\!{P}_{\infty}\|_{L^2}^{2}+\|H_t\|_{L^2}^2+\|\curl^2  H\|_{L^2}^2+\|\nabla H\|_{L^2}^2\right) ,
\end{align}
provide that $C_0<\hat{\varepsilon}_1\triangleq(4C_4+4C_5)^{-27}$.
By Gronwall's inequality, \eqref{tdh-4} and Lemmas \ref{lem-vn}, \ref{lem-basic}, one has
\begin{align}\label{bw23}
\displaystyle  \sup_{0\le t\le \si(T)}\|\nabla  u \|_{L^{2}}^{2}+\int_0^{\sigma(T)}\int\rho|\dot{u}|^{2}dxdt\leq C (\|\nabla u_0\|_{L^2}^2+\|\nabla H_0\|_{L^2}^2)+CC_0^{1/3},
\end{align}
which yields \eqref{uv1}.

It remains to prove \eqref{uv2}.  Taking $m=2-s$ in \eqref{J01}, \eqref{ht4},  and integrating over $(0,\sigma(T)]$ instead of $(0,T]$, in a similar way as we have gotten \eqref{a20}, we obtain
\begin{align}\label{bu2}
&\sup_{0\le t\le \sigma(T)}\sigma (\|\sqrt{\rho}\dot{u}\|_{L^2}^2+\|H_t\|_{L^2}^2)+\int_0^{\sigma(T)}\sigma (\|\nabla\dot{u}\|_{L^2}^2+\|\nabla H_t\|_{L^2}^2)dt \nonumber \\
&\leq C\int_0^{\sigma(T)}\sigma  \|\nabla u\|^4_{L^4}dt+ \tilde{C} ,
\end{align}
where we have taken advantage of \eqref{h2xd1} and \eqref{uv1}.
Furthermore, by \eqref{tdu2}, \eqref{h-tdh} and \eqref{uv1}, for $s\in (1/2,1]$, we have
\begin{align}\label{bu3}
&\quad\int_0^{\sigma(T)}t \|\nabla u\|_{L^{4}}^{4}dt \nonumber \\
& \le C\int_0^{\sigma(T)}t (\|\sqrt{\rho}\dot{u}\|_{L^{2}}^{3}+\|\curl^2  H\|_{L^2}^3+\|P-{P}_{\infty}\|_{L^{6}}^3+\|\nabla H\|_{L^2}^3)\|\nabla u\|_{L^{2}}dt \nonumber \\
& \quad +C\int_0^{\sigma(T)}t \|\nabla u\|_{L^{2}}^{4}dt \nonumber \\
& \le C+ C\int_0^{\sigma(T)}t^{-\frac{1}{4}}(t^{\frac{1}{4}}\|\nabla u\|_{L^2}^2)^{\frac{1}{2}}(t^{\frac{1}{4}}\|\sqrt{\rho}\dot{u}\|_{L^{2}}^{2})^{\frac{1}{2}}(t \|\sqrt{\rho}\dot{u}\|_{L^2}^2)dt \nonumber \\
& \quad+ C\int_0^{\sigma(T)}t^{-\frac{1}{4}}(t^{\frac{1}{4}}\|\nabla u\|_{L^2}^2)^{\frac{1}{2}}(t^{\frac{1}{4}}\|\curl^2  H\|_{L^{2}}^{2})^{\frac{1}{2}}(t \|\curl^2  H\|_{L^2}^2)dt\nonumber \\
& \le CC_0^{\frac{1}{9}}\sup_{0\le t\le \sigma(T)}(t (\|\sqrt{\rho}\dot{u}\|_{L^{2}}^{2}+\|\curl^2  H\|_{L^{2}}^{2}))+C.
\end{align}
Besides, from \eqref{h2xd1} and \eqref{uv1}, we have
\begin{align}\label{h2xd3}
& \quad \sup_{0\leq t\leq \sigma(T)}t  \|\curl^2  H\|_{L^2}^2
\leq  C\sup_{0\leq t\leq \sigma(T)} t \|H_t\|_{L^2}^2 +\tilde{C}.
\end{align}
Then combining this with \eqref{bu2} and \eqref{bu3}, we have
\begin{align}\label{bu4}
&\sup_{0\le t\le \sigma(T)}\sigma (\|\sqrt{\rho}\dot{u}\|_{L^2}^2+\|H_t\|_{L^2}^2)+\int_0^{\sigma(T)}\sigma (\|\nabla\dot{u}\|_{L^2}^2+\|\nabla H_t\|_{L^2}^2)dt \nonumber \\
&\leq C_6C_0^{\frac{1}{9}}\sup_{0\le t\le \sigma(T)}\sigma (\|\sqrt{\rho}\dot{u}\|_{L^2}^2+\|H_t\|_{L^2}^2)+ \tilde{C} ,
\end{align}
Therefore, if we choose $C_0$ to be such that $C_0\leq \ve_3 \triangleq \min\{\hat{\ve}_1, (2C_6)^{-9}\}$, \eqref{bu4} and \eqref{h2xd3} implies \eqref{uv2}. The proof of Lemma \ref{lem-a1} is completed.
\end{proof}

\begin{lemma}\label{lem-a3}
Under the conditions of Proposition \ref{pr1}. Then there exists a positive constant  $\varepsilon_4$   depending only on $\mu ,  \lambda , \nu,  \ga ,  a ,  {\rho}_{\infty}, \bar{\rho}$, $M_1$ and $M_2$ such that if $C_0<\varepsilon_4$,
\begin{align}\label{ba3}
\displaystyle  A_4(\sigma(T))+A_5(\sigma(T))\leq C_0^{\frac{1}{9}}.
\end{align}
\end{lemma}
\begin{proof} 
We begin with the estimate on $A_4(\sigma(T))$. Using \eqref{uv1}, we have
\begin{align}\label{ba4}
A_4({\sigma(T)}) \leq & \sup_{0\le t\le\sigma(T)} ( \|\nabla u\|\ltwo)^{\frac{3}{4}} \sup_{0\le t\le\sigma(T)} (t\|\nabla u\|\ltwo)^{\frac{1}{4}} \nonumber \\
& +\sup_{0\le t\le\sigma(T)} ( \|\nabla H\|\ltwo)^{\frac{3}{4}} \sup_{0\le t\le\sigma(T)} (t\|\nabla H\|\ltwo)^{\frac{1}{4}} \nonumber \\
& +\left(\int_0^{\sigma(T)}  \|\rho^{\frac{1}{2}}\dot{u}\|\ltwo dt \right)^{\frac{3}{4}} \left(\int_0^{\sigma(T)} t\|\rho^{\frac{1}{2}}\dot{u}\|\ltwo dt \right)^{\frac{1}{4}} \nonumber \\
& +\left(\int_0^{\sigma(T)}  \|\curl^2  H\|\ltwo dt \right)^{\frac{3}{4}} \left(\int_0^{\sigma(T)} t\|\curl^2  H\|\ltwo dt \right)^{\frac{1}{4}} \nonumber \\
\leq & CA_1(T)^{\frac{1}{4}}\leq CC_0^{\frac{1}{8}}\leq C_0^{\frac{1}{9}}.
\end{align}

Next, it remains to estimate $A_5({\sigma(T)})$. 
Multiplying $\eqref{CMHD}_2$ by $3|u|u$, and integrating over $ \O$ leads  to
\begin{align}\label{ba33}
\left(\int\rho|u|^{3}dx\right)_t 
&\leq C\int|u||\nabla u|^{2}dx+C\int|P-{P}_{\infty}||u||\nabla u|dx+C\int|H||\nabla H||u|^2dx \nonumber \\
&\leq C\|\nabla u\|_{L^2}^{\frac{5}{2}}(\|\rho\dot{u}\|_{L^2}^{\frac{1}{2}}+\|P-P_\infty\|_{L^2}^{\frac{1}{2}}+\|\curl^2  H\|_{L^2}^{\frac{1}{2}}+\|\nabla H\|_{L^2}^{\frac{1}{2}})\nonumber \\
&\quad +C\|\nabla u\|_{L^2}^3+CC_0^{\frac{1}{6}}\|\nabla u\|_{L^2}^2+C(\|\nabla H\|_{L^2}^4+\|\nabla u\|_{L^2}^4).
\end{align}
Hence, integrating \eqref{ba33} over $(0,\sigma(T))$ and using \eqref{key1}, \eqref{basic3}, we get
\begin{align}\label{ba34}
\sup_{0\le t\le  \si(T) }\int\rho|u|^{3}dx 
&\leq C\int_0^{\sigma(T)}(t^{\frac{1}{4}}\|\nabla u\|_{L^2}^2)^{\frac{5}{4}} (t^\frac{1}{4}(\|\rho\dot{u}\|_{L^2}^2+\|\curl^2  H\|_{L^2}^2))^{\frac{1}{4}} t^{-\frac{3}{8}} dt \nonumber \\
&\quad + C\int_0^{\sigma(T)}(t^{\frac{1}{4}}\|\nabla u\|_{L^2}^2)^{\frac{5}{4}} (t^\frac{1}{4}\|\nabla H\|_{L^2}^2)^{\frac{1}{4}} t^{-\frac{3}{8}} dt \nonumber \\
&\quad + CC_0+CC_0^{\frac{2}{9}}+\int\rho_0|u_0|^3dx \nonumber \\
&\leq CC_0^{\frac{1}{6}}+\int\rho_0|u_0|^3dx \leq C_7C_0^{\frac{1}{6}},
\end{align}
where we have used the fact 
\begin{align}\label{ba35}
 \displaystyle \int\rho_0|u_0|^{3}dx\leq C\|\rho_0^{\frac{1}{2}}u_0\|_{L^{2}}^{\frac32}\|\nabla u_0\|_{L^2}^{\frac32}\leq CC_0^{\frac32}.
 \end{align}
Finally, set $\varepsilon_4\triangleq\min\{\varepsilon_3,(C_7)^{-18}\}$, we get $A_5({\sigma(T)})\leq C_0^{\frac{1}{9}}$.
The proof of Lemma \ref{lem-a3} is completed.
\end{proof}

\begin{lemma}\label{lem-a1a2} Under the conditions of Proposition \ref{pr1}. Then there exists a positive constant $\varepsilon_5$ depending only  on $\mu,$  $\lambda,$ $\nu,$ $\gamma,$ $a$,   ${\rho}_{\infty}$, $\bar{\rho}$, $M_1$ and $ M_2$  such that
 \begin{align}\label{a1a2}
 \displaystyle  A_1(T)+A_2(T)\le C_0^{\frac{1}{2}},
 \end{align}
provided $C_0\leq\varepsilon_5$.
\end{lemma}

\begin{proof} By \eqref{g1} and \eqref{basic2}, one can check that
\begin{align}\label{al3}
\int_0^{T}\sigma \|\nabla u\|_{L^3}^3 dt
\le  C \int_0^{T}\sigma\|\nabla u\|_{L^2}\|\nabla u\|_{L^4}^2dt\le  C C_0+C \int_0^{T} \sigma^2\|\nabla u\|_{L^4}^4dt,
\end{align}
which, along with \eqref{a01} and \eqref{a02} gives
\begin{align}\label{a1a2-1}
\displaystyle  A_1(T)+A_2(T)\leq C(C_0^{\frac{31}{54}}+\int_0^T\sigma^2\|\nabla u\|_{L^{4}}^{4}dt).
\end{align}
So it reduces to estimate $\int_0^T\sigma^2\|\nabla u\|_{L^{4}}^{4}dt$.
On the one hand, by \eqref{tdu2}, \eqref{h-tdh} \eqref{key1} and Lemma \ref{lem-basic} again,  it indicates that
\begin{align}\label{al4-l}
&\quad\int_0^{\sigma(T)}t^{2}\|\nabla u\|_{L^{4}}^{4}dt \nonumber \\
& \le C\int_0^{\sigma(T)}t^{2}\big((\|\sqrt{\rho}\dot{u}\|_{L^{2}}^{3}\!+\|P\!-\!{P}_{\infty}\|_{L^{6}}^3\!+\|\nabla H\|_{L^2}^3\!+\|\curl^2  H\|_{L^2}^3)\|\nabla u\|_{L^{2}}\!+\|\nabla u\|_{L^{2}}^{4}\big)dt \nonumber \\
& \le C\int_0^{\sigma(T)}t^{-\frac{1}{4}}(t^{\frac{1}{4}}\|\nabla u\|_{L^2}^2)^{\frac{1}{2}}(t^{\frac{1}{4}}\|\sqrt{\rho}\dot{u}\|_{L^{2}}^{2})^{\frac{1}{2}}(t^{2}\|\sqrt{\rho}\dot{u}\|_{L^2}^2)dt+CC_0 \nonumber \\
& \quad+ C\int_0^{\sigma(T)}t^{-\frac{1}{4}}(t^{\frac{1}{4}}\|\nabla u\|_{L^2}^2)^{\frac{1}{2}}(t^{\frac{1}{4}}\|\curl^2  H\|_{L^{2}}^{2})^{\frac{1}{2}}(t^{2}\|\curl^2  H\|_{L^2}^2)dt\nonumber \\
& \le CC_0^{\frac{11}{18}}.
\end{align}
On the other hand, by \eqref{key1}, \eqref{basic2} and Lemma \ref{lem-a1}, we have
\begin{align}\label{al4-h}
&\quad\int_{\sigma(T)}^{T}\sigma^{2}\|\nabla u\|_{L^{4}}^{4}dt \nonumber \\
& \le C\int_{\sigma(T)}^{T}\sigma^{2}\big((\|\sqrt{\rho}\dot{u}\|_{L^{2}}^{3}\!+\|P\!-\!{P}_{\infty}\|_{L^{6}}^3\!+\|\nabla H\|_{L^2}^3\!+\|\curl^2  H\|_{L^2}^3)\|\nabla u\|_{L^{2}}\!+\|\nabla u\|_{L^{2}}^{4}\big)dt \nonumber \\
& \le CC_0+C\int_{\sigma(T)}^{T}\sigma^{2}\|P\!-\!{P}_{\infty}\|_{L^{4}}^4 dt.
\end{align}
Furthermore, it follows from $\eqref{CMHD}_1$ and \eqref{flux} that $P-P_\infty $ satisfies
\begin{align}\label{pp1}
&(P-P_\infty)_t+u\cdot\nabla (P-P_\infty)+\frac{\gamma}{2\mu+\lambda}(P-P_\infty)F \nonumber \\ 
 &\quad+\frac{\gamma}{2\mu+\lambda}(P-P_\infty)^2+\frac{\gamma}{2(2\mu+\lambda)}(P-P_\infty)|H|^2+\ga P_\infty{\rm div}u=0.
\end{align}
Multiplying \eqref{pp1} by $3 (P-P_\infty )^2$ and integrating over $\Omega,$ after using \eqref{f-lp}, we get
\begin{align}\label{pp2}
& \frac{3\gamma-1}{2\mu+\lambda}\|P-P_\infty \|_{L^4}^4 \nonumber \\ 
\leq&-\left(\|P-P_\infty \|_{L^3}^3\right)_t+\delta\|P-P_\infty \|_{L^4}^4+C\|F\|_{L^4}^4+C\|\nabla H\|_{L^2}^6+C\|\nabla u\|_{L^2}^2 \nonumber \\ 
\leq&-\left(\|P-P_\infty \|_{L^3}^3\right)_t+\delta\|P-P_\infty \|_{L^4}^4+C(\|\rho\dot{u}\|_{L^2}^3+\|\curl ^2H\|_{L^2}^3)(\|\nabla u\|_{L^2}+\|\nabla H\|_{L^2}) \nonumber \\ 
&+C(\|\rho\dot{u}\|_{L^2}^3+\|\curl ^2H\|_{L^2}^3)\|P-P_\infty \|_{L^2}+C\|\nabla H\|_{L^2}^6+C\|\nabla u\|_{L^2}^2. 
\end{align}
Multiplying \eqref{pp2} by $\si^2$, then integrating over $(0,T],$ and choosing $\delta$ suitably small, by \eqref{basic2},  \eqref{key1} and \eqref{basic3}, we obtain
\begin{align}\label{pp3}
&\int_0^{T} \sigma^2\|P-P_\infty \|_{L^4}^4 dt \nonumber \\ 
\le & C\sup_{0\le t\le T}\|P-P_\infty \|^3_{L^3}+C\int_0^{\si(T)}\|P-P_\infty \|^3_{L^3}dt \nonumber \\ 
&+C\int_0^{T} \sigma^2(\|\rho\dot{u}\|_{L^2}^3+\|\curl ^2H\|_{L^2}^3)(\|\nabla u\|_{L^2}+\|\nabla H\|_{L^2})dt \nonumber \\ 
 &+C\int_0^{T}\sigma^2(\|\rho\dot{u}\|_{L^2}^3+\|\curl ^2H\|_{L^2}^3)\|P-P_\infty \|_{L^2}dt+ CC_0^{\frac{13}{18}}+CC_0 \nonumber \\ 
\le &  C(\on)C_0^{\frac{11}{18}}. 
\end{align}
Combining \eqref{al4-l}, \eqref{al4-h} and \eqref{pp3}, it follows from \eqref{a1a2-1} that
\begin{align}\label{a1a2-2}
\displaystyle  A_1(T)+A_2(T)\leq C_{8}C_0^{\frac{11}{18}}.
\end{align}
Set $\varepsilon_5\triangleq\min\{\varepsilon_4,(C_{8}^{-9}\}$, then \eqref{a1a2} holds when $C_0<\varepsilon_5$. The proof of Lemma \ref{lem-a1a2} is completed.
\end{proof}

We now proceed to proof the uniform (in time) upper bound for the
density.

\begin{lemma}\label{lem-brho}
Under the conditions of Proposition \ref{pr1}. Then there exists a positive constant $\varepsilon_6$ depending only  on $\mu,$  $\lambda,$ $\nu,$ $\gamma,$ $a$, ${\rho}_{\infty}$, $\bar{\rho}$, $M_1$ and $ M_2$  such that
 \begin{align}\label{brho}
 \displaystyle  \sup_{0\le t\le T}\|\n(t)\|_{L^\infty}  \le
\frac{7\bar{\rho} }{4}  ,
 \end{align}
provided $C_0\le \ve_6. $
\end{lemma}

\begin{proof}
First, the equation of  mass conservation $\eqref{CMHD}_1$ can be equivalently rewritten in the form
\begin{align}\label{rho1}
\displaystyle  D_t \n=g(\rho)+b'(t),
\end{align}
where
\begin{align*}
 \displaystyle D_t\rho\triangleq\rho_t+u \cdot\nabla \rho ,\,
g(\rho)\triangleq-\frac{\rho(P-{P}_{\infty})}{2\mu+\lambda},\,
 b(t)\triangleq-\frac{1}{2\mu+\lambda} \int_0^t\rho (F+\frac{|H|^2}{2})dt.
 \end{align*}
Naturally, we shall prove our conclusion by Lemma \ref{lem-z}. It is sufficient to check that the function $b(t)$ must verify \eqref{a100} with some suitable constants $N_0$, $N_1$.

For $t\in[0,\sigma(T)],$ one deduces from \eqref{g1}, \eqref{g2}, \eqref{tdf1}, \eqref{tdxd-u1}, \eqref{key1} and Lemmas \ref{lem-basic}, \ref{lem-a1} that for $\delta_0$ as in Proposition \ref{pr1} and for all $0\leq t_1\leq t_2\leq\sigma(T)$,
\begin{align}\label{bl1}
&\quad |b(t_2)-b(t_1)| =\frac{1}{\lambda+2\mu}\left|\int_{t_1}^{t_2}\rho (F+\frac{|H|^2}{2})dt\right|\nonumber\\
&\le C\int_0^{\sigma(T)}(\|F\|_{L^{\infty}}+ \|H\|^2_{L^{\infty}})dt \nonumber\\
& \le C\int_0^{\sigma(T)}\|F\|_{L^{6}}^{\frac{1}{2}}\|\nabla F\|_{L^{6}}^{\frac{1}{2}}dt+C\int_0^{\sigma(T)}\|H\|_{L^{6}}\|\nabla H\|_{L^{6}} dt \nonumber\\
& \leq C\int_0^{\sigma(T)}(\|\sqrt{\rho}\dot{u}\|_{L^2}^{\frac{1}{2}}+\|\curl^2  H\|_{L^2}^{\frac{1}{2}})\|\nabla \dot{u}\|_{L^{2}}^{\frac{1}{2}}dt\nonumber\\
& \quad + C\int_0^{\sigma(T)}(\|\sqrt{\rho}\dot{u}\|_{L^2}^{\frac{1}{2}}+\|\curl^2  H\|_{L^2}^{\frac{1}{2}})\|\nabla H\|_{L^{2}}^{\frac{1}{4}}\|\curl^2  H\|_{L^{2}}^{\frac{3}{4}}dt\nonumber\\
& \quad + C\int_0^{\sigma(T)}(\|\sqrt{\rho}\dot{u}\|_{L^2}^{\frac{1}{2}}+\|\curl^2  H\|_{L^2}^{\frac{1}{2}})\|\nabla H\|_{L^{2}}dt\nonumber\\
& \quad + C\int_0^{\sigma(T)}(\|\nabla H\|_{L^2}^{\frac{1}{2}}\|\nabla \dot{u}\|_{L^{2}}^{\frac{1}{2}}+\|\nabla H\|_{L^{2}}^{\frac{3}{4}}\|\curl^2  H\|_{L^{2}}^{\frac{3}{4}}+\|\nabla H\|_{L^2}^{\frac{3}{2}})dt\nonumber\\
& \quad + C\int_0^{\sigma(T)}\|\nabla H\|_{L^2}(\|\nabla H\|_{L^2}+\|\curl^2  H\|_{L^2})dt \triangleq \sum_{i=1}^5 B_i.
\end{align}
We have to estimate $B_i, i=1,2,\cdots,5$ one by one. A directly computation gives
\begin{align}\label{bl-b1}
B_1 &\leq C\int_0^{\sigma(T)}\big(t^{\frac{1}{4}}(\|\sqrt{\rho}\dot{u}\|_{L^2}^2+\|\curl^2  H\|_{L^2}^2)\big)^{\frac{1}{4}}\big(t \|\nabla \dot{u}\|_{L^{2}}^2\big)^{\frac{1}{4}}t^{-\frac{5}{16}}dt\nonumber\\
&\leq CC_0^{\frac{1}{36}},
\end{align}
similarly,
\begin{align}
B_2 &\leq C\int_0^{\sigma(T)}\big(t (\|\sqrt{\rho}\dot{u}\|_{L^2}^2+\|\curl^2  H\|_{L^2}^2)\big)^{\frac{1}{4}}\big(t^{\frac{1}{4}}\|\nabla H\|_{L^{2}}^2\big)^{\frac{1}{8}}\nonumber \\
 & \qquad \times\big(t^{\frac{1}{4}}\|\curl^2  H\|_{L^{2}}^2\big)^{\frac{3}{8}}t^{-\frac{3}{8}}dt
 \leq CC_0^{\frac{1}{36}},\label{bl-b2}\\
B_3 &\leq C\int_0^{\sigma(T)}\big(t^{\frac{1}{4}}(\|\sqrt{\rho}\dot{u}\|_{L^2}^2+\|\curl^2  H\|_{L^2}^2)\big)^{\frac{1}{4}}\big(t^{\frac{1}{4}}\|\nabla H\|_{L^{2}}^2\big)^{\frac{1}{2}}t^{-\frac{3}{8}}dt\nonumber \\
 & \leq CC_0^{\frac{1}{12}},\label{bl-b3}\\
B_4 &\leq C\int_0^{\sigma(T)}\|\nabla H\|_{L^2}^{\frac{1}{2}}\big(t \|\nabla \dot{u}\|_{L^{2}}^2\big)^{\frac{1}{4}}t^{-\frac{1}{4}}dt\nonumber\\
&\quad +C\int_0^{\sigma(T)}\big(t^{\frac{1}{4}}\|\nabla H\|_{L^{2}}^2\big)^{\frac{3}{8}}\big(t^{\frac{1}{4}}\|\curl^2  H\|_{L^{2}}^2\big)^{\frac{3}{8}}t^{-\frac{3}{16}}dt \nonumber \\
 & \quad + C\int_0^{\sigma(T)}\big(t^{\frac{1}{4}}\|\nabla H\|_{L^2}^{2}\big)^{\frac{3}{4}}t^{-\frac{3}{16}}dt\leq CC_0^{\frac{1}{12}},\label{bl-b4}\\
B_5 &\leq \int_0^{\sigma(T)}\big(t^{\frac{1}{4}}\|\nabla H\|_{L^{2}}^2\big)^{\frac{1}{2}}\big(t^{\frac{1}{4}}\|\curl^2  H\|_{L^{2}}^2\big)^{\frac{1}{2}}t^{-\frac{1}{4}}dtdt+CC_0 \nonumber \\&
\leq CC_0^{\frac{5}{9}}.\label{bl-b8}
\end{align}
Putting \eqref{bl-b1}-\eqref{bl-b8} into \eqref{bl1}, we have
\begin{align}\label{bl2}
|b(t_2)-b(t_1)| \leq C_{9}C_0^{\frac{1}{36}}.
\end{align}
Combining \eqref{bl2} with \eqref{rho1} and choosing $N_1=0$, $N_0=C_{9}C_0^{\frac{1}{36}}$, $\bar{\zeta}={\rho}_{\infty}$ in Lemma \ref{lem-z} give
\begin{align}\label{rho2}
\displaystyle  \sup_{t\in [0,\si(T)]}\|\rho\|_{L^\infty} \le \bar{\rho}
+C_{9}C_0^{\frac{1}{36}} \le\frac{3\bar{\rho}}{2},
\end{align}
provided $C_0\le \hat{\ve}_6\triangleq\min\{\varepsilon_5, \left(\frac{\bar{\rho}}{2C_{9}}\right)^{36}\}. $

On the other hand, for $t\in[\sigma(T),T],\,\,\sigma(T)\le t_1\le t_2\le T ,$ it follows from \eqref{tdf1}, \eqref{key1}, and Lemma \ref{lem-basic} that
\begin{align}\label{br1}
& \quad |b(t_2)-b(t_1)| \le C\int_{t_1}^{t_2}(\|F\|_{L^{\infty}}+\|H\|^2_{L^{\infty}})dt \nonumber\\
&\le \frac{a}{\lambda+2\mu}(t_2-t_1)+C\int_{t_1}^{t_2}\|F\|_{L^{\infty}}^{8/3}dt +C\int_{t_1}^{t_2}\|H\|_{L^{\infty}}^2dt \nonumber\\
& \le \frac{a}{\lambda+2\mu}(t_2-t_1)+CC_0^{\frac{1}{6}}\int_{\sigma(T)}^{T}(\|\nabla \dot{u}\|_{L^2}^{2}+\|\nabla H\|_{L^2}\|\curl^2  H\|_{L^2}^3+\|\nabla H\|_{L^2}^4)dt\nonumber\\
&\quad +CC_0 +C\int_{t_1}^{t_2}(\|\nabla H\|_{L^2}\|\curl^2  H\|_{L^2}+\|\nabla H\|_{L^2}^2)dt \nonumber \\
& \le \frac{a}{\lambda+2\mu}(t_2-t_1)+C_{10}C_0^{2/3}.
\end{align}
Now we choose $N_0=C_{10}C_0^{2/3}$, $N_1=\frac{a}{\lambda+2\mu}$ in \eqref{a100} and set $\bar\zeta= \frac{3\bar{\rho}}{2}$ in \eqref{a101}. Since for all $  \zeta \geq\bar{\zeta}=\frac{3\bar{\rho}}{2}>{\rho}_{\infty}+1$,
$$ g(\zeta)=-\frac{ a\zeta}{2\mu+\lambda}(\zeta^{\gamma}-{\rho}_{\infty}^{\gamma})\le -\frac{a}{\lambda+2\mu}= -N_1. $$
Together with \eqref{rho1} and \eqref{br1}, by Lemma \ref{lem-z}, we have
\begin{align}\label{rho3}
\displaystyle \sup_{t\in
[\si(T),T]}\|\rho\|_{L^\infty}\le \frac{ 3\hat \rho }{2} +C_{10}C_0^{2/3} \le
\frac{7\hat \rho }{4} ,
\end{align}
provided $C_0\le \ve_6 \triangleq\min\{\hat{\ve}_6, (\frac{ \hat \n }{4C_{10}})^{3/2}\}$.
The combination of \eqref{rho2} with \eqref{rho3} completes the
proof of Lemma \ref{lem-brho}.
\end{proof}

\subsection{\label{se4} Time-dependent higher order estimates }
In this subsection, we derive the time-dependent higher order estimates, which are necessary for the global existence of classical solutions. The procedure is similar as that in \cite{cl2021,lxz2013,lx2016}, and we sketch it here for completeness.
From now on, assume that the initial energy $C_0\leq \ve_6$, and the positive constant $C $ may depend on $T,$ $\mu$, $\lambda$, $\nu$, $a$, $\ga$, ${\rho}_{\infty},$ $\bar{\rho},$   $\Omega$, $M_1, M_2$,  $\|\na u_0\|_{H^1}, \|\na H_0\|_{H^1}, \|\n_0-{\rho}_{\infty}\|_{W^{2,q}},$ $\|P(\n_0)-{P}_{\infty}\|_{W^{2,q}}$,  $\| g\|_{L^2}$ for $q\in(3,6)$ where $g\in L^2(\Omega)$ is given by compatibility condition \eqref{dt3}.

\begin{lemma}\label{lem-x1}
 There exists a positive constant $C,$ such that
\begin{align}
&\sup_{0\le t\le T}(\|\nabla u\|_{L^2}^2+\|\nabla H\|_{L^2}^2)+\int_0^T(\|\sqrt{\rho}\dot{u}\|_{L^2}^2+\| H_t\|_{L^2}^2+\|\nabla^2 H\|_{L^2}^2)dt\leq C,\label{x1b1}\\
&\sup_{0\le t\le T}(\|\sqrt{\rho}\dot{u}\|_{L^2}^2+\| H_t\|_{L^2}^2+\|\nabla^2 H\|_{L^2}^2)+\int_0^T(\|\nabla\dot{u}\|_{L^2}^{2}+\|\nabla H_t\|_{L^2}^2)dt\leq C,\label{x1b2}\\
&\sup_{0\le t\le T}(\|\nabla\rho\|_{L^6 \cap L^2}+\|\nabla u\|_{H^1})+\int_0^T(\|\nabla u\|_{L^\infty}+\|\nabla^{2} u\|_{L^6}^{2})dt\leq C.\label{x2b1}
\end{align}
\end{lemma}
\begin{proof}
First, combining \eqref{tdh-4} and \eqref{uv1} along with \eqref{tdhk} gives \eqref{x1b1}.
Then choosing $m=0$ in \eqref{J01} and \eqref{ht3}, integrating them over $(0,T)$, by \eqref{J0b1}, \eqref{x1b1} and the compatibility condition \eqref{dt3}, we have
\begin{align}\label{x1b4}
\displaystyle  &\quad\sup_{0\le t\le T}(\|\sqrt{\rho}\dot{u}\|_{L^2}^2+\| H_t\|_{L^2}^2+\|\nabla^2 H\|_{L^2}^2)+\int_0^T(\|\nabla\dot{u}\|_{L^2}^{2}+\|\nabla H_t\|_{L^2}^2)dt \nonumber\\
& \leq C+C\int_0^T (\|\sqrt{\rho}\dot{u}\|_{L^2}^3+\|\nabla^2 H\|_{L^2}^3+\|\nabla H\|_{L^2}^4\|\nabla^2 H\|_{L^2}^2) dt \nonumber \\
&\leq C+\frac{1}{2}\sup_{0\le t\le T}(\|\sqrt{\rho}\dot{u}\|_{L^2}^2+\|\nabla^2 H\|_{L^2}^2),
\end{align}
where we have also used Lemma \ref{lem-gn}, Lemma \ref{lem-f-td}, \eqref{h-tdh} and \eqref{h2xd1}, then we deduce \eqref{x1b2} from \eqref{x1b4}.

Next we want to prove \eqref{x2b1}. For $2\leq p\leq 6 ,$ $|\nabla\rho|^p$ satisfies
\begin{align*}
& (|\nabla\rho|^p)_t + \text{div}(|\nabla\rho|^pu)+ (p-1)|\nabla\rho|^p\text{div}u  \\
 &+ p|\nabla\rho|^{p-2}(\nabla\rho)^{\ast} \nabla u (\nabla\rho) +
p\rho|\nabla\rho|^{p-2}\nabla\rho\cdot\nabla\text{div}u = 0. 	
 \end{align*} 
Thus, by \eqref{tdf1}, it follows
\begin{align}\label{x2b3}
 (\|\nabla\rho\|_{L^p})_t 
&\le C(1+\norm[L^{\infty}]{\nabla u} )\norm[L^p]{\nabla\rho} +C\|\nabla F\|_{L^p}+C\|\nabla H \cdot H\|_{L^p} \nonumber\\
&\le C(1+\|\nabla u\|_{L^{\infty}})\|\nabla\rho\|_{L^p}+C\|\rho\dot{u}\|_{L^p}+C\|H\|_{L^{\infty}}\|\nabla H\|_{L^p}.
\end{align}
We deduce from Gagliardo-Nirenberg's inequality, \eqref{tdf1}, \eqref{tdxd-u1} and \eqref{x2b1} that
\begin{align}\label{x2b4}
\displaystyle  & \quad\|\div u\|_{L^\infty}+\|\omega\|_{L^\infty}\nonumber\\
&\le C(\|F\|_{L^\infty}+\|P-\bar{P}\|_{L^\infty}+\|H\|^2_{L^\infty})+\|\omega\|_{L^\infty} \nonumber\\
&\le C(\|F\|_{L^2}+\|\nabla F\|_{L^6}+\|\omega\|_{L^2}+\|\nabla \omega\|_{L^6}+\|P-\bar{P}\|_{L^\infty}+\|H\|_{L^{6}}\|\nabla H\|_{L^6}) \nonumber\\
&\le C(\|\nabla\dot{u}\|_{L^2}+1),
\end{align}
which together with Lemma \ref{lem-bkm} and \eqref{2tdu} indicates that
\begin{align}\label{x2b5}
\displaystyle \|\na u\|_{L^\infty } &\le C\left(\|{\rm div}u\|_{L^\infty }+
\|\omega\|_{L^\infty } \right)\ln(e+\|\na^2 u\|_{L^6 }) +C\|\na
u\|_{L^2} +C \nonumber\\
&\le C(1+\|\nabla\dot{u}\|_{L^2})\ln(e+\|\nabla\dot u\|_{L^2 } +\|\na \rho\|_{L^6}) \nonumber\\
&\le C(1+\|\nabla\dot{u}\|_{L^2}^2)+C(1+\|\nabla\dot{u}\|_{L^2})\ln(e+\|\nabla\rho\|_{L^6}) . 
\end{align}

Consequently, taking $p=6$ in \eqref{x2b3} leads to
\begin{align*}
(e+\|\nabla\rho\|_{L^6})_t\leq C[1+\|\nabla\dot{u}\|_{L^2}^2+(1+\|\nabla\dot{u}\|_{L^2})\ln(e+\|\nabla\rho\|_{L^6})](e+\|\nabla\rho\|_{L^6}),
\end{align*}
which can be rewritten as
\begin{align}\label{x2b7}
\displaystyle  \left(\ln(e+\|\nabla\rho\|_{L^6})\right)_t\leq C(1+\|\nabla\dot{u}\|_{L^2}^{2})+C(1+\|\nabla\dot{u}\|_{L^2})\ln(e+\|\nabla\rho\|_{L^6}).
\end{align}
By Gronwall's inequality and \eqref{x1b2}, we derive
\begin{align}\label{x2b8}
\displaystyle  \sup_{0\leq t\leq T}\|\nabla\rho\|_{L^{6}}\leq C .
\end{align}
Furthermore, by \eqref{x2b5} and \eqref{2tdu}, together with \eqref{x1b1} and \eqref{x1b2}, we have
\begin{align}\label{x2b9}
\displaystyle  \int_0^T\|\nabla u\|_{L^\infty}dt\leq C.
\end{align}
Similarly, taking $p=2$ in \eqref{x2b3}, by Gronwall's inequality, together with \eqref{x1b2} and \eqref{x2b9}, we obtain that
\begin{equation}\label{x2b10}
\displaystyle  \sup_{0\leq t\leq T}\|\nabla\rho\|_{L^{2}}\leq C.
\end{equation}
Moreover, combining \eqref{x1b1}, \eqref{x1b2}, \eqref{x2b8}, \eqref{x2b9} and \eqref{x2b10} yields
 \begin{align}\label{x2b11}
\displaystyle  \int_0^T\|\nabla^{2} u\|_{L^6}^{2}dt\leq C\, ,\quad\sup_{0\leq t\leq T}\|u\|_{H^{2}}\leq C.
\end{align}
This finishes the proof of Lemma \ref{lem-x1}.
\end{proof}

\begin{lemma}\label{lem-x3}
There exists a positive constant $C$ such that
\begin{align}
& \sup_{0\le t\le T}\|\sqrt{\rho}u_t\|_{L^2}^2 + \ia\int|\nabla u_t|^2dxdt\le C, \label{x3b}\\
&\sup_{0\le t\le T}(\|{\rho- {\rho}_{\infty}}\|_{H^2} +  \|{P- {P}_{\infty}}\|_{H^2})\le C,\label{x3bb}\\
& \sup\limits_{0\le t\le T}\left(
 \|{P\!-\!\! {P}_{\infty}}\|_{H^2}\!+\!
   \|\rho_t\|_{H^1}\!+\!\|P_t\|_{H^1}\right)
    \!+\!\! \int_0^T\!\!\left(\|\n_{tt}\|_{L^2}^2\!+\!\|P_{tt}\|_{L^2}^2\right)dt
\le C, \label{x4b} \\
& \sup\limits_{0\le t\le T}\sigma (\|\nabla u_t\|_{L^2}^2+\|\nabla H_t\|_{L^2}^2)
    + \int_0^T\sigma(\|\sqrt{\rho}u_{tt}\|_{L^2}^2+\|H_{tt}\|_{L^2}^2)dt
\le C.\label{x4bb}
\end{align}
\end{lemma}
\begin{proof} Based on Lemma \ref{lem-x1}, \eqref{x3b}-\eqref{x4b} can be obtained by the same way as that in \cite{cl2019}.
It remains to prove \eqref{x4bb}. Introducing the function
$$K(t)=(\lambda+2\mu)\int(\div u_t)^{2}dx+\mu\int|\omega_t|^{2}dx +\nu \int|\curl H_t|^{2}dx.$$
Since $u_t\cdot n = 0, H_t \cdot n=0$ on $\partial\Omega$, by Lemma \ref{lem-vn}, we have
\begin{align}\label{x4b6}
\displaystyle  \|\nabla u_t\|_{L^2}^2+\|\nabla H_t\|_{L^2}^2\leq C(\Omega)K(t).
\end{align}
Differentiating  $\eqref{CMHD}_{2,3}$  with respect to $t,$
\begin{align}\label{utt}\rho u_{tt}\!\!-\!(\lambda\!+\!2 \mu)\nabla \div u_t\!+\!\mu \nabla \!\times\! \omega_t\!=\!-\!\nabla\! P_t\!-\!\rho_t u_t\!-\!(\rho u\! \cdot\! \nabla\! u)_t\!+\!(H\! \cdot\! \nabla\! H\!-\!\nabla |H|^2/2)_t, \end{align}
and
\begin{align}\label{htt} \quad H_{tt}-\nu \nabla \times \curl H_t =(H \cdot \nabla u-u \cdot \nabla H-H\div u)_t, \end{align}
then multiplying \eqref{utt} by $2u_{tt}$, multiplying \eqref{htt} $2H_{tt}$ respectively, we obtain
\begin{align}\label{x4b7}
\displaystyle &\quad\frac{d}{dt}K(t)+2\int(\rho|u_{tt}|^2+|H_{tt}|^2)dx \nonumber \\
&=\frac{d}{dt}\Big(-\int\rho_t|u_t|^{2}dx-2\int\rho_tu\cdot\nabla u\cdot u_tdx+2\int P_t\div u_tdx \nonumber \\
&\qquad\quad -\int (2(H \otimes H)_t : \nabla u_t-|H|^2_t\div u_t dx)\Big)\nonumber \\
&\quad +\int\rho_{tt}|u_t|^{2}dx + 2\int(\rho_tu\cdot\nabla u)_t\cdot u_tdx-2\int\rho (u\cdot\nabla u)_t\cdot u_{tt}dx \nonumber\\
&\quad - 2\int P_{tt}\div u_tdx+\int (2(H \otimes H)_{tt} : \nabla u_t-|H|^2_{tt}\div u_t) dx\nonumber\\
&\quad +2\int (H \cdot \nabla u-u \cdot \nabla H- H \div u)_t \cdot H_{tt} dx\nonumber\\
&\triangleq\frac{d}{dt}K_0 + \sum\limits_{i=1}^6 K_i .
\end{align}
Let us estimate $K_i$, $i=0,1,\cdots, 6.$
We conclude from $\eqref{CMHD}_1$, \eqref{x1b2}, \eqref{x2b1}, \eqref{x3b}, \eqref{x4b}, \eqref{x4b6} and Sobolev's, Poincar\'{e}'s  inequalities that
\begin{align}
K_0 & \le \left|\int{\rm div}(\rho u)\,|u_t|^2dx\right|+C\norm[L^3]{\rho_t}\| u\|_{L^\infty}\|\nabla u\|_{L^2}\norm[L^6]{u_t}+C\|P_t\|_{L^2}\|\nabla u_t\|_{L^2} \nonumber \\
&\quad + C \|H\|_{L^\infty}\|H_t\|_{L^2}\|\nabla u_t\|_{L^2} 
\le \frac{1}{2}K(t)+C,\label{x4k0}\\
K_1 &\leq \left|\int_{ }\rho_{tt}\, |u_t|^2 dx\right|= \left|\int_{ }\div(\rho u)_t\,|u_t|^2 dx\right|= 2\left|\int_{ }(\rho_tu + \rho u_t)\cdot\nabla u_t\cdot u_tdx\right|\label{x4k1}\nonumber\\
& \le C\|\nabla u_t\|_{L^2}^2 K(t)+C\|\nabla u_t\|_{L^2}^2+C,\\
K_2&+K_3+K_4 
 \le C\norm[L^2]{\rho_{tt}}^2 + C\norm[L^2]{\nabla u_t}^2+\norm[L^2]{\rho^{{\frac{1}{2}}}u_{tt}}^2 +C\norm[L^2]{P_{tt}}^2 +C,\label{x4k2} \\
K_5 
 &\leq \frac{1}{2}\|H_{tt}\|_{L^2}^2+C\|H_{t}\|_{L^2}^2K(t)+C(\|\nabla H_t\|_{L^2}^2+\|\nabla u_t\|_{L^2}^2),\label{x4k5}\\
K_6 
 &\leq \frac{1}{2}\|H_{tt}\|_{L^2}^2+C(\|\nabla H_t\|_{L^2}^2+\|\nabla u_t\|_{L^2}^2).\label{x4k6}
\end{align}
Consequently, multiplying \eqref {x4b7} by $\sigma$, together with \eqref{x4k1}-\eqref{x4k6}, we get
\begin{align}\label{x4b8}
\displaystyle  &\quad\frac{d}{dt}(\sigma K(t)-\sigma K_0)+\sigma\int(\rho|u_{tt}|^{2}+|H_{tt}|^2)dx \nonumber \\
&\le C(1+\|\nabla u_t\|_{L^2}^2)\sigma K(t)+C(1+\|\nabla u_t\|_{L^2}^2+\|\nabla H_t\|_{L^2}^2+\|\rho_{tt}\|_{L^2}^2+\|P_{tt}\|_{L^2}^2),
\end{align}
By Gronwall's inequality, \eqref{x1b2}, \eqref{x3b}, \eqref{x4b} and \eqref{x4k0}, we derive that
\begin{align}\label{x4b9}
\displaystyle \sup_{0\le t\le T}(\sigma K(t))+\int_0^T\sigma(\|\sqrt{\rho}u_{tt}\|_{L^2}^2+\|H_{tt}\|_{L^2}^2)dt\le C .
\end{align}
As a result, by \eqref{x4b6}, we get \eqref{x4bb}. This finishes the proof .
\end{proof}

\begin{lemma}\label{lem-x5}
There exists a positive constant $C$ so that for any $q\in(3,6),$
\begin{align}
&  \sup_{t\in[0,T]}\left(\|\rho- {\rho}_{\infty}\|_{W^{2,q}} +\|P-{P}_{\infty}\|_{W^{2,q}}\right)\le C,\label{x5b}\\
& \sup_{t\in[0,T]} \si (\|\nabla u\|_{H^2}^2+\|\nabla H\|_{H^2}^2)\nonumber \\
 & \qquad  +\int_0^T \left(\|\nabla u\|_{H^2}^2+\|\nabla H\|_{H^2}^2 +\|\na^2 u\|^{p_0}_{W^{1,q}}+\si\|\na u_t\|_{H^1}^2\right)dt\le C,\label{x5bb}
\end{align}
where $p_0=\frac{9q-6}{10q-12}\in(1,\frac{7}{6}).$
\end{lemma}
\begin{proof}
Let's start with \eqref{x5bb}.  By Lemma \ref{lem-x1} and Poincar\'{e}'s, Sobolev's inequalities, one can check that
\begin{align}\label{x5b1}
\|\nabla (\n \dot u) \|_{L^2}&\le \||\nabla \n ||  u_t|  \|_{L^2}\!+\! \|\n\nabla   u_t  \|_{L^2}\! +\! \||\nabla \n|| u||\nabla u| \|_{L^2}\! +\! \|\n|\nabla  u|^2\|_{L^2}\!+\! \|  \n |u || \nabla^2 u| \|_{L^2}\nonumber \\
 &\le C+C\| \nabla   u_t  \|_{L^2}.
\end{align}
Consequently, together with \eqref{x4b} and Lemma \ref{lem-x1}, it yields
\begin{align}\label{x5b2}
\|\nabla^2 u\|_{H^1} &\le C (\|\rho \dot u\|_{H^1}+\|H \cdot \nabla H\|_{H^1}+ \| P-{P}_{\infty}\|_{H^2}+ \| |H|^2\|_{H^2}+\|u\|_{L^2})\nonumber \\
 &\le C+C \|\na  u_t\|_{L^2}.
\end{align}
It then follows from \eqref{x5b2}, \eqref{x2b1}, \eqref{x3b} and \eqref{x4bb} that
\begin{align}\label{x5b3}
\displaystyle \sup\limits_{0\le
t\le T}\si\|\nabla  u\|_{H^2}^2+\ia \|\nabla  u\|_{H^2}^2dt \le
 C.
\end{align}
Next, from \eqref{CMHD1}$_3$, \eqref{tdhk}, it follows
\begin{align}\label{x5h1}
\|\nabla^2 H\|_{H^1} 
&\le C (\|H_t\|_{H^1}+\|u \cdot \nabla H\|_{H^1}+ \|H \cdot \nabla u\|_{H^1}+ \|H \div u\|_{H^1}+\|\nabla H\|_{L^2})\nonumber \\
 &\le C+C \|\nabla H_t\|_{L^2}.
\end{align}
Similarly, from \eqref{x5b1}, \eqref{x1b1} and \eqref{x2b1}, we obtain
\begin{align}\label{x5h2}
\displaystyle \sup\limits_{0\le
t\le T}\si\|\nabla  H\|_{H^2}^2+\ia \|\nabla  H\|_{H^2}^2dt \le
 C.
\end{align}
 Next, we deduce from Lemma \ref{lem-x1} and \eqref{x4b} that
\begin{align}\label{x5b4}
\displaystyle  \|\na^2u_t\|_{L^2}
&\le C(\|(\rho\dot{u})_t\|_{L^2}+\|\nabla P_t\|_{L^2}+\|((\nabla \times H)\times H)_t\|_{L^2}+\|u_t\|_{L^2}) \nonumber \\
&\le C\|\n^{\frac{1}{2}}  u_{tt}\|_{L^2}+C\|\nabla  u_t\|_{L^2} +C\|\nabla  H_t\|_{L^2}+C,
\end{align}
where in the first inequality, we have utilized the $L^p$-estimate for the following elliptic system
\begin{equation}\label{x5b5}
\begin{cases}
  \mu\Delta u_t+(\lambda+\mu)\nabla\div u_t=(\rho\dot{u})_t+\nabla P_t+((\nabla \times H)\times H)_t \,\,\, &\text{in} \,\,\Omega,\\
  u_t\cdot n=0\,\,\,\text{and} \,\,\,\omega_t\times n=0\,\,&\text{on} \,\,\partial\Omega.
\end{cases}	
\end{equation}
Together with \eqref{x5b4} and \eqref{x4bb} yields
\begin{align}\label{x5b6}
\displaystyle  \int_0^T\sigma\|\nabla u_t\|_{H^1}^2dt\leq C.
\end{align}

By Sobolev's inequality, \eqref{x2b1}, \eqref{x4b} and \eqref{x4bb}, we get for any $q\in (3,6)$,
\begin{align}\label{x5b7}
\displaystyle \|\na(\n\dot u)\|_{L^q}
&\le C \|\na \n\|_{L^q}(\|\nabla\dot{u}\|_{L^q}+\|\nabla\dot{u}\|_{L^2}+\|\nabla u\|_{L^2}^2)+C\|\na\dot u \|_{L^q}\nonumber\\
&\le C\sigma^{-\frac{1}{2}}+C\|\nabla u\|_{H^2}+C\sigma^{-\frac{1}{2}}(\sigma\|\nabla u_t\|_{H^1}^2)^{\frac{3(q-2)}{4q}}+C.
\end{align}
Integrating this inequality over $[0,T],$ by \eqref{x1b2} and \eqref{x5b6}, we have
\begin{align}\label{x5b8}
\displaystyle  \int_0^T\|\nabla(\rho\dot{u})\|_{L^q}^{p_0}dt\leq C .
\end{align}

On the other hand, \eqref{x4b} gives
\begin{align}\label{x5b9}
\displaystyle (\|\na^2 P\|_{L^q})_t & \le C \|\na u\|_{L^\infty} \|\na^2 P\|_{L^q}   +C  \|\na^2 u\|_{W^{1,q}}   \nonumber \\
& \le C (1+\|\na u\|_{L^\infty} )\|\na^2 P\|_{L^q}+C(1+ \|\na  u_t\|_{L^2})+ C\| \na(\n \dot u )\|_{L^{q}},
\end{align}
where in the last inequality we have used the  following simple fact that
\begin{align}\label{x5b10}
\displaystyle \|\na^2 u\|_{W^{1,q}}
 &\le C(1 + \|\na  u_t\|_{L^2}+ \| \na(\n\dot u )\|_{L^{q}}+\|\na^2  P\|_{L^{q}}),
\end{align}
due to \eqref{2tdu}, \eqref{3tdu}, \eqref{x1b2} and \eqref{x4b}.

Hence, applying Gronwall's inequality in \eqref{x5b9}, we deduce from \eqref{x2b1}, \eqref{x3b}  and \eqref{x5b8} that
\begin{align}\label{x5b11}
\displaystyle  \sup_{t\in[0,T]}\|\nabla^{2}P\|_{L^q}\leq C ,
\end{align}
which along with \eqref{x3b}, \eqref{x4b}, \eqref{x5b10} and \eqref{x5b8} also gives
\begin{align}\label{x5b12}
\displaystyle  \sup_{t\in[0,T]}\|P-{P}_{\infty}\|_{W^{2,q}}+\int_0^T\|\nabla^{2}u\|_{W^{1,q}}^{p_0}dt\leq C .
\end{align}
Similarly, one has
\begin{align}
\displaystyle \sup\limits_{0\le t\le T}\|
\n-{\rho}_{\infty}\|_{W^{2,q}} \le
 C,
\end{align}
which together with \eqref{x5b12}  gives \eqref{x5b}. The proof of Lemma \ref{lem-x5} is finished.
\end{proof}

\begin{lemma}\label{lem-x6}
There exists a positive constant $C$ such that, for any $q\in (3,6)$,
\begin{align}\label{x6b}
\displaystyle & \sup_{0\le t\le T}\sigma\left(\|\rho^{\frac{1}{2}} u_{tt}\|_{L^2}+\|H_{tt}\|_{L^2}+\|\nabla u_t\|_{H^1}+\|\nabla H_t\|_{H^1}+\|\nabla^2 H\|_{H^2}+\|\na u\|_{W^{2,q}}\right)\nonumber \\
& +\int_{0}^T\sigma^2(\|\nabla u_{tt}\|_{2}^2+\|\nabla H_{tt}\|_{2}^2)dt\le C.
\end{align}
\end{lemma}

\begin{proof} Differentiating $\eqref{CMHD}_{2,3}$ with respect to $t$ twice,
multiplying them by $2u_{tt}$ and $2H_{tt}$ respectively, and integrating over $\Omega$ lead to
\begin{align}\label{x6b2}
&\quad \frac{d}{dt}\int(\rho |u_{tt}|^2+|H_{tt}|^2)dx \nonumber \\
 &\quad +2(\lambda+2\mu)\int(\div u_{tt})^2dx+2\mu\int|\omega_{tt}|^2dx+2\nu\int|\curl H_{tt}|^2dx \nonumber \\
&=-8\int_{ }  \n u^i_{tt} u\cdot\na
 u^i_{tt} dx-2\int_{ }(\n u)_t\cdot \left[\na (u_t\cdot u_{tt})+2\na
u_t\cdot u_{tt}\right]dx \nonumber \\
&\quad -2\int_{}(\n_{tt}u+2\n_tu_t)\cdot\na u\cdot u_{tt}dx-2\int (\n
u_{tt}\cdot\na u\cdot  u_{tt}-P_{tt}{\rm div}u_{tt})dx \nonumber \\
&\quad -2\int (H \cdot \nabla H-\nabla |H|^2/2)_{tt}u_{tt}dx+2 \int(H \cdot \nabla u- u \cdot \nabla H- H \div u)_{tt}H_{tt}dx \nonumber \\
&\triangleq\sum_{i=1}^6 R_i.
\end{align}
Let us estimate $R_i$ for $i=1,\cdots,6$. H\"{o}lder's inequality and \eqref{x2b1} give
\begin{align}\label{x6r1}
\displaystyle  R_1 &\le
C\|\sqrt{\rho}u_{tt}\|_{L^2}\|\na u_{tt}\|_{L^2}\| u \|_{L^\infty}
\le \de \|\na u_{tt}\|_{L^2}^2+C(\de)\|\sqrt{\rho}u_{tt}\|^2_{L^2} .
\end{align}
By \eqref{x1b2}, \eqref{x3b}, \eqref{x4b} and \eqref{x4bb}, we conclude that
\begin{align}
 R_2
&\le \de \|\na u_{tt}\|_{L^2}^2+C(\de)\|\nabla u_t\|_{L^2}^3+C(\de)\|\nabla u_t\|_{L^2}^2,\label{x6r2}\\
 R_3 
&\le \de \|\na u_{tt}\|_{L^2}^2+C(\de)\|\n_{tt}\|_{L^2}^2+C(\de)\|\nabla u_t\|_{L^2}^2,\label{x6r3}\\
 R_4 
&\le \de \|\na u_{tt}\|_{L^2}^2+C(\de)\|\sqrt{\rho}u_{tt}\|^2_{L^2}
+C(\de)\|P_{tt}\|^2_{L^2},\label{x6r4}\\
 R_5 
&\le \de \|\na u_{tt}\|_{L^2}^2+C(\de)\|H_{tt}\|^2_{L^2}+C(\de)\|\nabla H_{t}\|^3_{L^2},\label{x6r5}\\
 R_6 
 &\le \de (\|\na u_{tt}\|_{L^2}^2+\|\na H_{tt}\|_{L^2}^2)+C(\de)\|H_{tt}\|^2_{L^2} \nonumber \\
 & \quad +C(\de)(\|\nabla u_{t}\|_{L^2}\|\nabla H_{t}\|^2_{L^2}+\|\nabla u_{t}\|^2_{L^2}\|\nabla H_{t}\|^2_{L^2}),\label{x6r6}
\end{align}

Substituting \eqref{x6r1}-\eqref{x6r6} into \eqref{x6b2}, utilizing the fact that
\begin{align}\label{x6b3}
\displaystyle  \|\nabla u_{tt}\|_{L^2}\leq C(\|\div u_{tt}\|_{L^2}+\|\omega_{tt}\|_{L^2}), \quad \|\nabla H_{tt}\|_{L^2}\leq C\|\curl H_{tt}\|_{L^2},
\end{align}
and then choosing $\de$ small enough, we can get
\begin{align}\label{x6b4}
&\frac{d}{dt}(\|\sqrt{\rho}u_{tt}\|^2_{L^2}+\|H_{tt}\|^2_{L^2})+\|\na u_{tt}\|_{L^2}^2+\|\na H_{tt}\|_{L^2}^2 \nonumber \\
&\le C (\|\sqrt{\rho}u_{tt}\|^2_{L^2}+\|H_{tt}\|^2_{L^2}+\|\rho_{tt}\|^2_{L^2}+\|P_{tt}\|^2_{L^2}+\|\nabla u_{t}\|^3_{L^2}+\|\nabla H_{t}\|^3_{L^2}) \nonumber \\
& \quad+C(\|\nabla u_{t}\|_{L^2}\|\nabla H_{t}\|^2_{L^2}+\|\nabla u_{t}\|^2_{L^2}\|\nabla H_{t}\|^2_{L^2}),
\end{align}
which together with \eqref{x4b}, \eqref{x4bb}, and by Gronwall's inequality yields that
\begin{align}\label{x6b5}
&\quad\sup_{0\leq t\leq T}\sigma^2(\|\sqrt{\rho}u_{tt}\|^2_{L^2}+\|H_{tt}\|^2_{L^2})+\int_0^T\sigma^2(\|\na u_{tt}\|_{L^2}^2+\|\na H_{tt}\|_{L^2}^2)dt 
\leq C.
\end{align}
Furthermore, it follows from \eqref{tdhk}, \eqref{3tdu}, \eqref{x5b4} and \eqref{x4bb} that
\begin{align}\label{x6b6}
\displaystyle &\quad \sup_{0\le t\le T}(\sigma\|\nabla^2 u_t\|_{L^2}+\sigma\|\nabla^2 H_t\|_{L^2}) \nonumber \\
 &\leq C \sigma(1+\|\rho^{\frac{1}{2}}u_{tt}\|_{L^2}+\|H_{tt}\|_{L^2}+\|\nabla u_t\|_{L^2}+\|\nabla H_t\|_{L^2}) \leq C.
\end{align}
Finally, we deduce from \eqref{x4bb}, \eqref{x5b}, \eqref{x5bb}, \eqref{x5h1}, \eqref{x5b7}, \eqref{x5b10}, \eqref{x6b5} and \eqref{x6b6} that
\begin{align}\label{x6b7}
&\displaystyle \quad \sigma\|\na^2 u\|_{W^{1,q}}
 \le C\sigma (1+\|\na  u_t\|_{L^2}+\|\na  H_t\|_{L^2}+\| \na(\n
\dot u )\|_{L^{q}}+\|\na^2  P\|_{L^{q}})\nonumber \\
& \le C(1+ \sigma\|\na u\|_{H^2}+\sigma^{\frac{1}{2}}(\sigma\|\na u_t\|_{H^1}^2)^{\frac{3(q-2)}{4q}})
\le C+C\sigma^{\frac{1}{2}}(\sigma^{-1})^{\frac{3(q-2)}{4q}} \le C ,
\end{align}
and
\begin{align}
\displaystyle  \sigma \|\nabla^2 H\|_{H^2}\leq C \sigma(1+\|\nabla H_t\|_{H^1}+\|\nabla u\|_{H^2}\|\nabla H\|_{H^2})\leq C,
\end{align}
together with \eqref{x6b5} and \eqref{x6b6} yields \eqref{x6b} and finishes the proof.
\end{proof}

\section{Proof of  Theorem  \ref{th1}-\ref{th2}}\label{se5}
In this section, we are prepared to proof the main results of this paper. Based on the estimates in Section \ref{se3}, we follows the procedure in \cite{lx2016,cl2021} to give the sketch of the proof.

{\it\textbf{Proof of Theorem \ref{th1}.} }
 By Lemma \ref{lem-local}, there exists a $T_*>0$ such that the  system \eqref{CMHD}-\eqref{boundary} has a unique classical solution $(\rho,u,H)$ on $\Omega\times
(0,T_*]$. One may use the a priori estimates, Proposition \ref{pr1} and Lemmas \ref{lem-x3}-\ref{lem-x6} to extend the classical solution $(\rho,u,H)$ globally in time.

First, by the definition of \eqref{As1}-\eqref{As5}, the assumption of the initial data \eqref{dt2} and \eqref{ba35}, one immediately checks that
\begin{align}\label{pf1}
\displaystyle  0\leq\rho_0\leq \bar{\rho},\,\, A_1(0)+A_2(0)=0, \,\,  A_3(0)\leq C_0^{\frac{1}{9}},\,\,A_4(0)+A_5(0)\leq C_0^{\frac{1}{9}}.
\end{align}
Therefore, there exists a $T_1\in(0,T_*]$ such that
\begin{equation}\label{pf2}
\begin{cases}
0\leq\rho_0\leq2\bar{\rho}, \,\, A_1(T)+A_2(T)\leq 2C_0^{\frac{1}{2}}, \\
A_3(T)\leq 2C_0^{\frac{1}{9}}, \,\, A_4(\sigma(T))+A_5(\sigma(T))\leq 2C_0^{\frac{1}{9}},
\end{cases}	
\end{equation}
hold for $T=T_1.$
Next, we set
\begin{align}\label{pf3}
\displaystyle  T^*=\sup\{T\,|\,{\rm \eqref{pf2} \ holds}\}.
\end{align}
Then $T^*\geq T_1>0$. Hence, for any $0<\tau<T\leq T^*$
with $T$ finite, it follows from Lemmas \ref{lem-x1}-\ref{lem-x6}
that
\begin{equation}\label{pf4}
\begin{cases}
\rho-{\rho}_{\infty} \in C([0,T]; H^2 \cap W^{2,q}), \\
(\nabla u,\nabla H) \in C([\tau ,T]; H^1),\quad ( \nabla u_t, \nabla H_t) \in C([\tau ,T]; L^q);
\end{cases}	
\end{equation}
where one has taken advantage of the standard embedding
$$L^\infty(\tau ,T;H^1)\cap H^1(\tau ,T;H^{-1})\hookrightarrow
C\left([\tau ,T];L^q\right),\quad\mbox{ for any } q\in [2,6).  $$
Due to \eqref{x3b}, \eqref{x4bb}, \eqref{x6b} and $\eqref{CMHD}_1$,
we obtain
\begin{align*}
&\quad\int_{\tau}^T \left|\left(\int\n|u_t|^2dx\right)_t\right|dt
\le\int_{\tau}^T\left(\|  \n_t  |u_t|^2 \|_{L^1}+2\|  \n  u_t\cdot u_{tt} \|_{L^1}\right)dt\\
&\le C\int_{\tau}^T\left( \| \n^{\frac{1}{2}} |u_t|^2 \|_{L^2}\|\na u\|_{L^\infty}+\|  u\|_{L^6}\|\na\n\|_{L^2} \|u_t  \|^2_{L^6}+  \|\sqrt{\rho}u_{tt} \|_{L^2}\right)dt\le C,
\end{align*}
which together with \eqref{pf4} yields
\begin{align}\label{pf5}
\displaystyle  \rho^{\frac{1}{2}}u_t, \quad\rho^{\frac{1}{2}}\dot u \in C([\tau,T];L^2).
\end{align}
Finally, we claim that
\begin{align}\label{pf6}
 \displaystyle  T^*=\infty.
 \end{align}
Otherwise, $T^*<\infty$. Then by Proposition \ref{pr1}, it holds that
\begin{equation}\label{pf7}
\begin{cases}
0\leq\rho\leq\frac{7\bar{\rho}}{4},\,\,A_1(T^*)+A_2(T^*)\leq C_0^{\frac{1}{2}},\\
A_3(T^*)\leq C_0^{\frac{1}{9}}, \,\, A_4(\sigma(T^*))+A_5(\sigma(T^*))\leq C_0^{\frac{1}{9}},	
\end{cases}	
\end{equation}
It follows from Lemmas \ref{lem-x5}, \ref{lem-x6} and \eqref{pf5} that $(\rho(x,T^*),u(x,T^*), H(x,T^*))$ satisfies the initial data condition \eqref{dt1}-\eqref{dt2}, \eqref{dt3}, where  $g(x)\triangleq\sqrt{\rho}\dot u(x, T^*),\,\,x\in \Omega.$
Thus, Lemma \ref{lem-local} implies that there exists some $T^{**}>T^*$ such that \eqref{pf2} holds for $T=T^{**}$, which contradicts the definition of $ T^*.$ As a result, \eqref{pf6} holds.
By Lemmas \ref{lem-local} and \ref{lem-x1}-\ref{lem-x6}, it indicates that $(\rho,u,H)$ is in fact the unique classical solution defined on $\Omega\times(0,T]$ for any  $0<T<T^*=\infty.$

Finally, with \eqref{m2}, \eqref{tdu1}, \eqref{tdh-2}, \eqref{grho}, \eqref{tdh-3} and \eqref{I01} at hand, \eqref{esti-t}  can be obtained in similar arguments as used in \cite{cl2021}, and we omit the details.
The proof of Theorem \ref{th1} is finished.   \endproof

{\it\textbf{Proof of Theorem \ref{th2}.} }
As is shown by \cite{cl2021}, we sketch the proof for completeness. First, we show that, for $T>0$, the Lagrangian coordinates of the system are given by
  \be \la{c61}  \begin{cases}\frac{\partial}{\partial \tau}X(\tau; t,x) =u(X(\tau; t,x),\tau),\,\,\,\, 0\leq \tau\leq T\\
 X(t;t,x)=x, \,\,\,\, 0\leq t\leq T,\,x\in\bar{\Omega}.\end{cases}\ee
 By \eqref{esti-uh}, the transformation \eqref{c61} is well-defined. In addition, by $\eqref{CMHD}_1$, we have
 \be\la{c62}\ba
\rho(x,t)=\rho_0(X(0; t, x)) \exp \{-\int_0^t\div u(X(\tau;t, x),\tau)d\tau\}.
\ea \ee
 If there exists some point $x_0\in \Omega$ such that $\n_0(x_0)=0,$ then there is a point $x_0(t)\in \bar{\Omega}$ such that $X(0; t, x_0(t))=x_0$. Hence, by \eqref{c62}, $\rho(x_0(t),t)\equiv 0$ for any $t\geq 0.$

 Now we will prove Theorem \ref{th2} by contradiction. Suppose there exist some positive constant $C_1$ and a subsequence ${t_{n_j}},$ $t_{n_j}\rightarrow \infty$ as $j\rightarrow \infty$ such that $\|\na\n (\cdot,t_{n_j})\|_{L^{r_1}}<C_1$. Consequently, by \eqref{g2} and the assumption $\rho_\infty>0$, we get that for $ r_1\in  (3,\infty)$ and $\theta=\frac{r_1-3}{2r_1-3}$,
\be\la{c63}\ba\rho_\infty\leq\|\rho(\cdot,t_{n_j})-\rho_\infty\|_{C\left(\ol{\O }\right)} \le C
\|\rho(\cdot,t_{n_j})-\rho_\infty\|_{L^3}^{\theta}\|\na \rho(\cdot,t_{n_j})\|_{L^{r_1}}^{1-\theta},
\ea\ee
which is in contradiction with \eqref{esti-t}. The proof is completed. \endproof






\section*{Acknowledgements} This research was partially supported by National Natural Sciences Foundation of China (Nos. 12171024, 11901025, 11971217, 11971020).

\end{document}